\documentclass{amsart}
\usepackage{amssymb}
\usepackage{amsbsy}
\usepackage{amscd}
\usepackage{amsmath}
\usepackage{amsthm}
\usepackage{amsxtra}
\usepackage{latexsym}

\usepackage[mathscr]{eucal}
\usepackage{mathrsfs}
\usepackage{upgreek}
\usepackage{wasysym}
\usepackage[all,cmtip]{xy}
\usepackage{xcolor}
\definecolor{rouge}{rgb}{0.85,0.1,.4}
\definecolor{bleu}{rgb}{0.1,0.2,0.9}
\definecolor{violet}{rgb}{0.7,0,0.8}
\usepackage[colorlinks=true,linkcolor=bleu,urlcolor=violet,citecolor=rouge]{hyperref}

\theoremstyle{theorem}
\newtheorem{thm}{Theorem}[section]
\newtheorem{lemma}[thm]{Lemma}
\newtheorem{claim}[thm]{Claim}
\newtheorem{prop}[thm]{Proposition}
\newtheorem{cor}[thm]{Corollary}
\theoremstyle{remark}
\newtheorem{defi}[thm]{Definition}
\newtheorem{rem}[thm]{Remark}

\newtheorem{ex}[thm]{Example}
\newtheorem{conj}[thm]{Conjecture}

\theoremstyle{theorem}

\newtheorem{theo_alpha}{Theorem}
\def\theo_alpha{\Alph{theo_alpha}}

\def\theo_alpha{\Alph{theo_alpha}}

\def\theo_alpha{\Alph{theo_alpha}}

\def\le{\leqslant}
\def\ge{\geqslant}

\def\F{\mathscr F}
\def\V{\mathcal V}
\def\X{\mathscr X}

\def\D{\mathscr D}
\def\O{\mathscr O}
\def\K{\mathscr K}

\def\Z{\mathbb Z}
\def\Q{\mathbb Q}
\def\N{\mathbb N}

\def\C{\mathbb C}

\def\P{\mathbb P}

\def\tP{\tilde{P}}
\def\T{T}
\def\d{{d}}

\def\diag{\mathrm{diag}}
\def\CC{\mathscr C}

\DeclareMathOperator{\Hom}{Hom}

\DeclareMathOperator{\Spec}{Spec}

\newcommand{\isomap}{\stackrel{\sim\,}{\longrightarrow}}

\begin{document}

\author[Victor Batyrev]{Victor Batyrev\textsuperscript{1}}
\address{\textsuperscript{1}Victor Batyrev, Mathematisches 
Institut, Universit\"at T\"ubingen,
72076 T\"ubingen, Germany}
\email{victor.batyrev@uni-tuebingen.de}

\author[Anne Moreau]{Anne Moreau\textsuperscript{2}}
\address{\textsuperscript{2}Laboratoire Paul Painlev\'{e}\\ 
Universit\'{e} Lille 1
\\ 59655 Villeneuve d'Ascq Cedex\\ France}
\email{anne.moreau@univ-lille.fr}

\title{Satellites of spherical subgroups}

\subjclass{14L30,14M27}
\keywords{Spherical homogeneous space, satellite, 
homogeneous spherical data, arc space, Poincar\'{e} polynomial.}

\date{\today}

\begin{abstract}
Let $G$ be a complex connected reductive algebraic group. 
Given a   
spherical subgroup $H \subset G$ 
and a subset $I$ of the set of spherical roots of $G/H$, we   
define,  up to conjugation, a 
spherical subgroup
$H_I \subset G$ of the same dimension of $H$, called a satellite. 
We investigate various interpretations of the satellites. 
We also show a close relation between the Poincar\'{e} polynomials of the two
spherical homogeneous spaces $G/H$ and $G/H_I$.  
\end{abstract}

\maketitle

\section{Introduction} \label{S:intro}
Let $G$ be a complex connected reductive algebraic group. 
An irreducible algebraic $G$-variety $X$ is said to be {\em spherical}
if $X$ is normal and if a Borel subgroup of $G$ 
has an open orbit in $X$. An algebraic subgroup
$H\subset G$ is called {\em spherical} if the homogeneous space $G/H$ is spherical. 
In this article, we introduce and explore new combinatorial invariants 
attached to spherical subgroups of $G$ from several perspectives. 

From now on, we fix a Borel subgroup $B$ of $G$ 
and a spherical subgroup  $H$ of~$G$. 
We first collect a few standard facts about the spherical homogeneous space 
$G/H$. 
There is no 
loss of generality in assuming 
that $BH$ is Zariski dense in $G$.
A {\em spherical embedding} of $G/H$ is a normal $G$-variety $X$
together with a $G$-equivariant open embedding $j\,:\,
G/H \hookrightarrow X$ 
which allows us to identify $H$ with the stabilizer of the point
$x:= j(H) \in X$ such that $Gx \cong G/H$.
Therefore we will always regard $G/H$ as the open $G$-orbit in $X$
whenever we write $(X,x)$ for a $G/H$-embedding.

Let $U:= BH/H$ be the open dense $B$-orbit in $G/H$. 
The {\em colors} of $G/H$ are the 
irreducible components of the complement of the open set $U$ in $G/H$, 
that is, the irreducible
$B$-invariant divisors in $G/H$. 
Denote by ${\mathscr D}={\mathscr D}(G/H)$ the set of colors of $G/H$. 
The stabilizer of
$U$ in $G$ is a parabolic subgroup $P$ of $G$ containing $B$. 
Then $U \cong B/B\cap H$ is an affine variety isomorphic to $(\C^*)^r \times \C^s$,
where $s$ is the dimension of the unipotent radical $P^{u}$ of $P$. 
The number $r$ is the {\em rank} of the spherical homogeneous 
space $G/H$. Let $\mathscr{M}(U)$ be the free abelian group of rank $r$ consisting of
all invertible regular functions $f$ in the affine coordinate ring $\C[U]$ such that
$f(H) =1$. Any such regular function $f \in \mathscr{M}(U)$ is a $B$-eigenfunction
associated with some weight $\omega(f) \in {\mathscr X}^*(B)$, 
where ${\mathscr X}^*(B)$ is the lattice of characters 
of the Borel subgroup $B$. The map $f \mapsto \omega(f)$ 
yields a natural embedding of the lattice
$\mathscr{M}(U)$ into the lattice ${\mathscr X}^*(B)$.
The {\em weight lattice} of $G/H$~is 
$$M= {\mathscr X}^*(G/H):= \omega(\mathscr{M}(U)).$$ 
We denote by $\langle - ,- \rangle \, : \, M \times  N \to \Z$ the natural pairing, 
where $N:=\Hom(M,\Z)$ is the dual lattice. 

A spherical embedding $X$ of $G/H$ is said to be {\em elementary} if
it consists of exactly two $G$-orbits: a dense orbit $X^0$ isomorphic to $G/H$
and a closed orbit $X'$
of codimension one. In this case $X$ is always smooth and it is uniquely determined
by the restriction  of the divisorial valuation $v_{X'}\, :\, \C(G/H) \to \Z$ to the
free abelian group $\mathscr{M}(U)  \subset \C(G/H)$,
where
$\C(G/H)$ is the field of rational functions on $G/H$, \cite[\S\S8.3, 7.5, 8.10]{LV}.
Using the isomorphism $\omega : \mathscr{M}(U) \isomap M$ we may 
view this restriction as 
an element $n_{X'} \in N$.
This induces a bijection between the set $\V=\V(G/H)$
of $\Q$-valued discrete
$G$-invariant valuations of $\C(G/H)$ and a convex polyhedral cone in
$N_\Q:=N\otimes_\Z \Q$, 
see \cite[Proposition 7.4]{LV} or \cite[Corollary~1.8]{K}. 
The latter is usually referred to as the {\em valuation cone} of
the spherical homogeneous space $G/H$.
In this way we may identify $\V$ with the valuation cone of $G/H$.
According to the Luna-Vust theory \cite[Sect.~8]{LV},
any $G$-equivariant embedding $X$ of $G/H$ can be combinatorially described by a {\em colored fan}
in the $r$-dimensional vector space $N_\Q$. 
The classification of $G$-equivariant embeddings (the Luna-Vust theory) 
was first studied in a more general setting. For spherical
varieties, this was later simplified 
and extended to any characteristic 
by Knop in \cite{K}. 

In this classification, elementary embeddings correspond 
to uncolored cones $\Q_{\ge 0} v$ 
in $N_\Q$ with $v$ a nonzero lattice point in $N \cap \V$. 
We shall denote by $(X_v,x_v)$ the elementary embedding 
of $G/H$, with closed $G$-orbit $X'_v$, corresponding 
to a nonzero lattice point $v$ in $N \cap \V$.  
It is uniquely defined up to $G$-equivariant isomorphism. 

It is known that $\V$ is a cosimplicial cone, i.e., there exist $k \le r$ linearly independent
primitive lattice vectors $s_1, \ldots, s_k \in M$, 
called the {\em spherical roots} of $G/H$, such that
\[ \V = \{ n \in N_\Q \; | \; \langle s_i, n \rangle \le 0, \; 1 \le i \le k \}.\] 
The definition of spherical roots goes back to Luna \cite[\S1.2]{Luna01}. 
The set $\Sigma=\Sigma(G/H)$ 
of spherical roots is one of the components of the {\em homogeneous spherical data 
associated with $G/H$}, a fundamental combinatorial invariant \cite[\S2.1]{Luna01}. 
The spherical roots $s_1, \ldots, s_k$ are nonnegative integral linear 
combinations of the positive roots 
with respect to the Borel subgroup $B \subset G$ \cite[Sect.~4]{BP}.

\subsection{Main illustrating example}  \label{sub:mot}
Symmetric homogeneous spaces are  
homogeneous spaces for which the stabilizer of a typical point 
is the fixed point set of a nontrivial group involution. 
It is a standard fact that they are spherical (\cite{Vust74}). 
The lattice $N=\Hom(M,\Z)$ of a symmetric 
homogeneous space $G/H$ is naturally endowed with a root system, 
and the spherical roots of $G/H$ can be expressed by means of this 
root system (\cite[\S2]{Vust90}). 
The classification of all embeddings 
of symmetric homogeneous spaces was carried out in 
 \cite{Vust90}. 

In order to illustrate  the content of our paper, we consider the well-known example 
(see e.g.,~\cite[Example 4.1]{Br97}) 
of the group 
$G$ viewed  as a symmetric homogeneous space under $G \times G$ acting 
by left and right multiplication: $(a,b)c:=acb^{-1}$, $a,b,c \in G$.
The stabilizer of the base point 1 is $\Delta(G)$, where $\Delta$ is 
the diagonal closed immersion. 
Let $B^-$ be the unique Borel subgroup such that 
$T:=B^- \cap B$ is a maximal torus of $G$, 
and let $S$ be the set of simple roots of $G$ relatively to $(B,T)$. 
The spherical roots of $G$ (relatively to the Borel subgroup $B^- \times B$)  
are the pairs $(-\alpha,\alpha)$,  where 
$\alpha \in S$.  
Thus $M=\mathscr{X}^*(G)$ 
is identified with the root lattice of $S$, 
and the valuation cone is the antidominant Weyl chamber. 

With every subset $I \subset  S$ 
one associates a spherical subgroup $\Delta(G)_I \subset G\times G$ 
as follows. 
Let 
$P_I$ and $P_I^-$ be the two opposite parabolic subgroups containing 
$B$ and $B^-$, respectively, such
that their common Levi component $L_I:= P_I^- \cap P_I$ 
has $I$ as a set of simple roots. 
Then set 
\[ \Delta(G)_I:= \left(P_I^{u} \times (P_I^-)^{u}\right) \Delta(L_I),\]
where $P_I ^{u}$ and $(P_I ^-)^{u}$ 
are the unipotent radicals of $P_I$ and $P_I^-$, respectively. 
We notice that $(B^- \times B) \Delta(G)_I$ is Zariski dense in $G\times G$.
We call the subgroups $\Delta(G)_I$, where 
$I\subset  S$, the {\em satellites of} $\Delta(G)$.
All satellites of $\Delta(G)$ have the same dimension
$\dim G=\dim \Delta(G)$. Since $(B^- \times B)\cap \Delta(G) = 
(B^- \times B) \cap \Delta(G)_I$ for all $I \subset S$,
the weight lattices of the satellites are all equal to $M$: 
\[ M=\mathscr{X}^*(G) = {\mathscr X}^*((G \times G)/\Delta(G)_I), 
\quad \text{for all} \quad I \subset S.\]
We refer to Example \ref{exa3:arc} for another approach to this example 
for $G=GL_n$. 

When $G$ is adjoint, that is, with trivial center, 
de Concini and Procesi 
constructed a {wonderful compactification} 
of the symmetric homogeneous space $G \cong (G\times G)/\Delta(G)$ 
\cite[\S3.4]{Concini-Procesi83}. 
The datum of the satellites $\{\Delta(G)_I \, ; \, I \subset S\}$ 
allows describing the stabilizers of points in the wonderful compactification 
(cf.~\S\ref{sub:toroidal_compactifications}) and computing  its Poincar\'{e} polynomial 
(cf.~\S\ref{sub:Poincare} 
and Example~\ref{exa:P}).

Our purpose is to generalize the above construction 
to an arbitrary spherical homogeneous space $G/H$.

\subsection{Brion subgroups} \label{sub:Brion}
Choose a nonzero primitive lattice point $v$ in the valuation cone
$\V \subset N_\Q$. 
The total space $\tilde{X}_v$ of the normal bundle 
to the $G$-invariant
divisor $X'_v$ in the elementary embedding $X_v$ 
has a natural $G$-action. 
One can show that
$ \tilde{X}_v$ is again a spherical
$G$-variety containing exactly two $G$-orbits: the zero section of the normal bundle
(it is isomorphic to $X'_v$) and its open complement $\tilde{X}_v^0$. 
The stabilizer $H_v$ of
a point in the open $G$-orbit  $\tilde{X}_v^0$ will be referred to as the 
{\em Brion subgroup} corresponding to $v \in \V$. Such a subgroup
was first studied by Brion in
\cite[\S1.1]{B90}\footnote{In \cite{B90}, $H$ is assumed to be
equal to its normalizer, but this assumption is superfluous in the construction of
the subgroup $H_v$.}.
It is defined up to $G$-conjugation. 
The above construction of the subgroup $H_v$ will be explained in greater detail  
in Sect.~\ref{sec:def}. 
In Sect.~\ref{sec:alg}, we indicate an algebraic description of Brion subgroups. 

We can alternatively define the Brion subgroup $H_v$ by considering
the total space $\tilde{X}_v^{\vee}$ of the conormal
bundle of $X'_v$
in $X_v$. To be more specific, if $(\tilde{X}_v^{\vee})^0$ stands for the open complement
of the zero section of $\tilde{X}_v^{\vee}$, then
$(\tilde{X}_v^{\vee})^0 \cong G/H_v$ and $(\tilde{X}_v^{\vee})^0$
is spherical, \cite[\S2.1]{Pa}.
Furthermore, $\tilde{X}_v^\vee$ is an elementary spherical embedding of
$G/H_v$.
Note that $\tilde{X}_v^{\vee}$ has a natural symplectic structure whose induced
$G$-action is Hamiltonian.

If $v=0$, we simply declare the Brion subgroup $H_v$ to be $H$. 

Our main observation is the following (see Theorem \ref{Th:inv-data} 
for a more precise formulation):
\begin{thm} \label{Th:inv-intro}
The Brion subgroup
$H_v \subset G$
only depends, up to conjugation,  on the minimal face $\V(v)$ 
of the valuation cone $\V$ containing $v$.
\end{thm}

Theorem \ref{Th:inv-intro} 
is proved in Sect.~\ref{sec:Luna}.  It follows  
from the description of the {\em homogeneous spherical data} 
introduced by Luna \cite{Luna01}
associated with the homogeneous space $G/H_v$ (see Theorem \ref{Th:inv-data}).
Set
\[I(v):= \{ s_i \in \Sigma \; | \;
\langle s_i, v \rangle =  0 \}.\]
Then the minimal face $\V(v)$ of the valuation cone $\V$ containing $v$
is determined by the equations $\langle s_i, v \rangle =  0$, 
where $s_i$ runs through the set $I(v)$.

In particular, if $v$ is an interior lattice point of the
valuation cone $\V$, then \hbox{$I(v) =\varnothing$} 
and the Brion subgroup $H_v$ is known to be
{\em horospherical},
i.e., $H_v$ contains a maximal unipotent subgroup of $G$, \cite[Corollary 3.8]{BP} 
(see also Proposition~\ref{pro:sat}~(2)).
The horospherical
satellite $H_{\varnothing}$ of $H$ is closely related
to a flat deformation of the spherical homogeneous space $G/H$
to the horospherical one $G/H_{\varnothing}$
(it has already appeared in the papers of Alexeev-Brion \cite{AM} 
and Kaveh \cite{Kav}, using the previous work of Popov \cite{Popov87}).  

We remark that for any subset $I \subset \Sigma$ there exists a primitive lattice
point $v \in \V$ such that $I(v)=I$, unless $I= \Sigma$ and $\Sigma$
is a basis of $M_\Q:=M \otimes_\Z \Q$.
This provides a natural reversing bijection between the  
faces of the valuation cone $\V$
and the subsets $I \subset  \Sigma$.
Therefore the conjugacy class of a Brion subgroup $H_v$ only depends
on the subset $I(v)$.
Given a subset $I \subset \Sigma$, $I\not=\Sigma$, we define up to conjugacy
the {\em spherical satellite} $H_I$ of $H$ as a Brion subgroup $H_v$, where $I= I(v)$.
For $I=\Sigma$ we simply set $H_{\Sigma}:=H$.

\subsection{Stabilizers in toroidal compactifications} \label{sub:toroidal_compactifications} 
A spherical embedding $X$ of $G/H$ is called {\em simple} if $X$ has a unique
closed $G$-orbit, 
and a simple spherical embedding $X$ is called {\em toroidal} 
if no color $D \in {\mathscr D}$ 
contains the unique closed $G$-orbit in its closure.  
Thus a spherical embedding $X$ is simple and toroidal 
if and only if its fan has no color and contains a unique cone of maximal dimension. 

If the number $k= |\Sigma|$ of spherical roots of $G/H$ is equal to the rank $r$
(i.e., the maximal possible number of spherical roots), then the valuation cone 
$\V$ is strictly convex (i.e., it contains no line). 
If so, 
according to the Luna-Vust theory,
$G/H$ admits a canonical normal projective $G$-equivariant 
simple toroidal 
compactification $X$ such
that we have a natural bijection $I \leftrightarrow X_I$ between the 
subsets $I \subset \Sigma$
and the $G$-orbits $X_I \subset X$. We show that the normalizer 
of the stabilizer of any point 
in the $G$-orbit $X_I$ is equal, up to $G$-conjugacy, to the 
normalizer $N_G(H_I)$ of the spherical satellite $H_I$
(cf.~Proposition~\ref{Pro:norm}).
This stabilizer is exactly $N_G(H_I)$ if $I=\varnothing$.
In particular, the unique closed $G$-orbit $X_{\varnothing}$ is the smooth projective
homogeneous $G$-variety $G/N_G(H_{\varnothing})$. 

It is well-known that in the case where $G/H$ is horospherical, 
for each spherical embedding $X$ of $G/H$, 
a conjugate of the stabilizer of any point $x \in X$ contains 
the subgroup $H$. In \cite[Proposition 2.4]{BM},  
such stabilizers are explicitly described. 
Our notion of satellites allows us to extend this fact 
to the more general case of spherical varieties. 
More precisely, since all satellites $H_I$ $(I\subset \Sigma)$
have the same weight lattice
$M$ (cf.~Proposition \ref{pro:sat} (1) or Proposition~\ref{pro:alg}), we may 
view every lattice
point $m \in M$ as a character of the normalizer $N_G(H_I)$ of $H_I$ in $G$
(cf.~\cite[Theorem 4.3]{Br97} or Lemma \ref{lem:R}), and
we obtain:

\begin{thm} \label{Th:stab}
Let $X$ be a simple toroidal spherical embedding of $G/H$
corresponding to  an uncolored cone $\sigma$ in $N_\Q$.
Denote by  $I(\sigma)$ the set of all spherical
roots in $\Sigma$ that vanish on
$\sigma$. Choose 
a point $x'$ in the unique closed $G$-orbit $X'$ of $X$, 
 and let $G_{x'}$ be its stabilizer 
in $G$. Then, up to a conjugation, we have the inclusions:
\[ H_{I(\sigma)} \subset G_{x'} \subset N_G(H_{I(\sigma)}). \]
Moreover, there is a homomorphism (given by Lemma \ref{lem:R})
from $N_G(H_{I(\sigma)})$ to the torus $\Hom(\sigma^\perp \cap M,\C^*)$
whose kernel is $G_{x'}$.
\end{thm}

Since every simple spherical embedding of $G/H$ is dominated by
a simple toroidal one, it follows that the stabilizer of a point in any
spherical embedding of $G/H$ contains, up to conjugation, one of the spherical
satellites $H_I$ (cf.~Corollary~\ref{cor:stab}). 

\subsection{Limits of stabilizers of points in arc spaces}
Let $\mathscr{O}:=\C[[t]]$ be the ring of formal power series, and
let $\mathscr{K}:=\C((t))$ be its field of fractions. We write 
$(G/H)(\K)$ for the set of $\K$-valued points of the spherical homogeneous
space $G/H$, and $G(\O)$ for the group of $\O$-valued points of the reductive group
$G$. 
Luna and Vust 
established  in \cite[\S4]{LV} 
a natural bijection between
lattice points $v$ in $\V$ and  $G(\O)$-orbits in $(G/H)(\K)$.
One can choose a
representative $\widehat{\lambda}_v$ of the $G(\O)$-orbit 
attached to $v$ 
using an appropriate 
one-parameter subgroup  in $G$. 
The reader is referred to Sect.~\ref{sec:arc} for more details. 

\begin{thm}  \label{Th:arc}
The Brion subgroup $H_v$ consists of limits when $t$ goes to $0$
of elements in the stabilizer 
 in $G(\O)$ of $\widehat{\lambda}_v$.
\end{thm}

Theorem \ref{Th:arc} is helpful in practice to compute the satellites, 
cf.~Example \ref{exa3:arc} and 
Example \ref{exa4:arc}. Its proof is performed in Sect.~\ref{sec:arc}. 

\subsection{Poincar\'{e} polynomial of satellites} \label{sub:Poincare}
To every complex algebraic variety $X$ we associate its {\em
virtual Poincar\'{e} polynomial} $P_X (t)$, uniquely determined by the following properties: 
\begin{enumerate}
\item $P_X (t) = P_Y (t)+ P_{X \setminus Y} (t)$ for every closed subvariety $Y$ 
of $X$,
\item if $X$ is smooth and complete, then $P_X (t) =\sum_{m}
\dim H^m(X) t^m$ is the usual Poincar\'{e} polynomial.
\end{enumerate}
Then $P_X(t)=P_Y(t) P_F(t)$ for every fibration $
F \hookrightarrow X \to Y$ which is locally trivial for the Zariski topology. 
The existence of the virtual Poincar\'{e} polynomial for every complex algebraic variety $X$ 
follows from the existence of Deligne's mixed Hodge structure on the 
cohomology groups $H_c^\bullet(X)$ of $X$ with compact supports and 
complex coefficients, which yields a polynomial $E_X(s,t)$ in two variables 
(see for example \cite[\S3]{BatyrevDais} or \cite[\S1]{DanilovKhovanskii}). 
We have $P_X(t)=E_X(-t,-t)$. 

In \cite{BPe}, Brion and Peyre investigated  
the virtual Poincar\'{e} polynomials  
of homogeneous spaces 
under complex connected linear
algebraic groups.  
In particular, they showed that they are 
polynomials in~$t^2$. They also  
obtained a factorization result 
for the virtual Poincar\'{e} polynomial of {\em regular embeddings} 
in the sense of \cite{BifetConciniProcesi} 
of such homogeneous 
spaces, provided that the stabilizer of a typical point 
is connected. 

Accordingly, if $X$ is a spherical homogeneous space $G/H$, 
then the function $P_X(t^{1/2})$ is a polynomial in $t$. 
By additivity, if $X$ is a spherical embedding of a 
spherical homogeneous space $G/H$, 
then $P_X(t^{1/2})$ is a polynomial since $X$ 
is a disjoint finite union of $G$-orbits which are all spherical homogeneous spaces. 
Therefore, for our purpose, it will be convenient to set: 
$$\tilde{P}_X(t) := P_X(t^{1/2}).$$
For example, $\tilde{P}_X(t)=(t-1)^r t^s$ if $XÊ\cong (\C^*)^r \times \C^s$. 
Abusing of notations, we continue to call the function $\tilde{P}_X$ 
the {\em virtual Poincar\'{e} polynomial} of $X$. 

A $G$-variety $X$ is said to be {\em wonderful} (cf.~Introduction of~\cite{Luna96}) 
of rank $r$ if it is complete and smooth,
with an open $G$-orbit whose complement is the union of $r$ smooth irreducible 
$G$-divisors $Z_1, \ldots, Z_r$, 
any subset of these irreducible $G$-divisors has a transversal and non-empty intersection, 
and these intersections are exactly all $G$-orbit closures of $X$.

A wonderful embedding of $G/H$ (that is, a $G/H$-embedding 
which is a wonderful $G$-variety) is unique, up to a $G$-equivariant isomorphism,  
if it exists 
(see e.g., \cite[\S1.3]{Luna01} or
\cite[Sect.~30]{T}), and this happens  
for instance if $H=N_G(H)$ in which case $X$ is 
the canonical  
toroidal embedding, \cite[Corollary 7.2]{K96} and \cite[Intro.]{Luna96}). 
Wonderful embeddings were first introduced by de Concini and Procesi 
\cite{Concini-Procesi83} for
symmetric spaces, motivated by problems in enumerative geometry.

Assume that $G/H$ admits a {wonderful compactification} $X$ 
of rank $r$.  
Then the spherical homogeneous space $G/H$ has rank $r$ and 
$X$ contains exactly $2^r$ $G$-orbits $X_I$
that are parameterized by all possible subsets $I \subset \{1, \ldots, r\}$ such that 
the closure $\overline{X_I}$ is the intersection of all irreducible 
$G$-divisors $Z_i$ with $i \not\in I$. 

As the wonderful embedding $X$ is the disjoint union $\bigsqcup_{I} X_I$ 
we have 
\begin{align*}
\tP_X(t) = \sum_I \tP_{X_I}(t).
\end{align*}
On the other hand, for $I \subset \{1,\ldots,r\}$, 
the total space $\tilde{X}_I$  of the normal bundle 
to the smooth subvariety $\overline{X_I} \subset X$ 
 is a 
simple 
spherical toroidal embedding of the spherical homogeneous space $G/H_I$. 
It is direct sum of line bundles. Indeed, 
since $\overline{X_I}= \bigcap_{i \not\in I} Z_i$,  
the normal bundle $N_{\overline{X_I}|X}$ of $X$ along $\overline{X_I}$ 
is the direct sum of 
the restrictions to $\overline{X_I}$ of the line bundles  
$N_{Z_i|X}$, where $i \not\in I$.  
One has a natural locally trivial 
torus fibration $f_I\, :\, G/H_I \to X_I$ for the Zariski topology 
whose fiber is isomorphic 
to $(\C^*)^{r-|I|}$. 
Hence, we get  
\begin{align} \label{eq:Poincare_wonderful}
\tP_X(t) = \sum_{I} \frac{\tP_{G/H_I}(t)}{(t-1)^{r - | I | }}.
\end{align}
Because the valuation cone $\V \subset N_\Q$ of $G/H$ is generated
by a basis $e_1, \ldots, e_r$ of the lattice $N$, every lattice point $v \in  
N\cap \V$ can be written as a nonnegative integral linear combination 
$v = l_1 e_1 + \cdots + l_r e_r$. 
The set of 
spherical roots $\{ s_1, \ldots, s_r\} \subset M$ forms a dual basis to the basis 
$\{e_1, \ldots, e_r\}$ of the dual lattice $N$. 
Let $\V^\circ_I$ be the 
relative interior of the face $\V_I$ defined by 
the conditions $\langle x, s_i \rangle =0$ for all $i \in I$. 
Using the power expansion
\[ \frac{1}{t-1} = \frac{t^{-1}}{(1 - t^{-1})} = \sum_{j\ge 1} t^{-j}, \]
we can rewrite \eqref{eq:Poincare_wonderful} 
as 
\begin{align} \label{eq:P-wonderful}
\tP_X(t) = \sum_I \frac{\tP_{G/H_I}(t)}{(t-1)^{r - | I |}} = 
\sum_I \tP_{G/H_I}(t) \sum_{v \in \V^\circ_I} t^{\kappa(v)},  
\end{align}
where $\kappa$ is the linear function on $N$ which takes value $-1$ on the lattice vectors
$e_1, \ldots, e_r$. 

Formula \eqref{eq:P-wonderful} is an analog of the one
we established for the Poincar\'{e} polynomial of a smooth projective 
toroidal horospherical variety
(see \cite{BM}). The general version of both formulas for the {stringy} 
$E$-function of an arbitrary $\Q$-Gorenstein 
spherical embedding $X$ will appear in \cite{BM2}. It provides   
a motivic interpretation 
of the Brion-Peyre factorization result (cf.~\cite[Theorem~2]{BPe}) 
for the virtual 
Poincar\'{e} polynomial of $X$. 
This 
has motivated us to investigate the ratio of the two virtual 
Poincar\'{e} polynomials, 
\[  \frac{\tP_{G/H_I}(t)}{\tP_{G/H}(t)}, \] 
for general spherical homogeneous space $G/H$. 
Looking at some series of examples we came to the following conjecture:

\begin{conj} \label{conj:P}
Let $G/H$ be an arbitrary spherical homogeneous space. Then for any 
satellite $H_I$ of the spherical subgroup $H\subset G$, the ratio of the 
two virtual 
Poincar\'{e} polynomials
\[   \frac{\tP_{G/H_I}(t)}{\tP_{G/H}(t)} \] 
is always a polynomial $R_I$ in $t^{-1}$ with integral coefficients. 
\end{conj}

We prove this conjecture in several examples, notably in the case 
where the spherical group $H$ is connected (cf.~Theorem \ref{Th2:Poincare}) 
and in the case where the spherical homogeneous 
space $G/H$ is of rank one (cf.~Theorem \ref{Th:P-one}).

\subsection*{Notations}
Unless otherwise specified, we keep the notations used in Sect.~\ref{S:intro} 
 further on in the article.

\subsection*{Acknowledgments}
We are very grateful to Michel Brion for the many valuable 
discussions we have had on the topic. 
We also thank Syu Kato, Shrawan Kumar, Dmitri Panyushev, Pierre Schapira 
and Ernest Vinberg for helpful explanations,  
and Guido Pezzini 
for the counter-example in the footnote \ref{foot:counter_example}. 
Some parts of this work were done while 
the first named author 
was visiting the University of Poitiers, 
while the second named author was visiting the 
University of T\"{u}bingen and during our stay  
as ^^ ^^ Research in Pairs'' at MFO. 
We thank all these institutes for their hospitality. 
Finally, we would like to thank the referee for the 
extremely careful reading of the paper and for his comments 
and detailed suggestions which helped us to considerably improve 
the text. 

The second named author is supported in part by the ANR Project GeoLie Grant number 
ANR-15-CE40-0012, and in part by the Labex CEMPI (ANR-11-LABX-0007-01). 

\section{Brion subgroups, first properties and examples} \label{sec:def} 
The section starts with a number of results of \cite{LV}, \cite{BLV} and \cite{BP} 
concerning elementary embeddings. 
Recall that $P^{u}$ stands for the unipotent radical of the parabolic subgroup $P$. 

\begin{defi}[{\cite[\S4.2]{BLV} or \cite[\S2.9]{BP}}] 
\label{def:adapted}
Let $L$ be a Levi subgroup of $P$, and let $C$ be the neutral component
of its center. We say that $L$ is {\em $H$-adapted} if
the following conditions are satisfied:
\begin{enumerate}
\item $P\cap H=L\cap H$,
\item $P\cap H$ contains the derived subgroup $(L,L)$ of $L$,
\item for any elementary embedding $(X,x)$ of $G/H$,
the action of $P^{u}$ on $Y:=Px \cup Px'$, with $Px'$ the open
$P$-orbit in the closed orbit of $X$, induces an isomorphism
of algebraic varieties $P^{u} \times (\overline{C x}\cap Y) \to Y$.
\end{enumerate}
\end{defi}
The existence of $H$-adapted Levi subgroups  
is established in \cite[\S3 and \S4.2]{BLV}. 
Fix such an $H$-adapted Levi subgroup $L$.  
Since $P$ has an open dense orbit in $G/H$, any $f\in \mathscr{M}(U) \cong M$ is 
determined by its weight $\omega(f) \in \X^*(P)$,  
$f$  
being a $P$-eigenvector in $\C(G/H)$. 
Furthermore, the restriction to $P\cap H$ of $\omega(f)$ 
is trivial. But the sublattice of $\X^*(P)$ consisting of 
characters whose restriction to $P\cap H$ is trivial 
identifies with the lattice of characters $\X^*(\mathbb{T})$ 
of the algebraic torus $\mathbb{T}:=C/C\cap H$ since 
$P=P^{u} L$ and $H$ contains $(L,L)$. 
Here $C$ denotes, as in the above definition, the neutral component
of the center of $L$. 
As a result, we get that $M \cong \X^*(\mathbb{T})$. 
Hence, by duality, the lattice 
$N_\Q\cong \Hom(M,\Q)$ identifies 
with $\X_*(\mathbb{T}) \otimes_\Z \Q$, 
where $\X_*(\mathbb{T}) $ is the free abelian group of 
one-parameter subgroups of~$\mathbb{T}$. 

\begin{defi} 
\label{def:param}
A one-parameter subgroup $\lambda$ of $C$ is said to be {\em adapted to 
the elementary embedding $(X,x)$} if $\lim\limits_{t\to 0}\lambda(t)x$ exists 
and belongs to
the open $P$-orbit of the closed orbit. 
\end{defi}
Through the identification $N_\Q \cong \X_*(\mathbb{T}) \otimes_\Z \Q$, 
the primitive lattice points in $N\cap \V$ are in bijection 
with the indivisible one-parameter subgroups 
of $\mathbb{T}$, adapted to the different elementary embeddings of $G/H$. 
More precisely, if $v$ is a nonzero lattice point in $N\cap \V$, 
then any one-parameter subgroup 
of $C$ adapted to $(X_v,x_v)$ 
corresponds to a point 
of $N\cap \V$,  
which is equivalent to $v$, cf.~\cite[\S2.10]{BP}.

We now give 
a concrete construction of the Brion subgroups $H_v$, 
following~\cite[\S1.1]{B90}.
Fix a nonzero lattice point  $v \in N \cap \V$ 
and choose a 
one-parameter subgroup $\lambda_v$ of $C$ adapted to $(X_v,x_v)$.
Thus
$$x'_v:=\lim_{t\to 0}\lambda_v(t)x_v$$
belongs to the closed $G$-orbit $X'_v$. 
The stabilizer $H'_v$ of $x'_v$ in $G$ 
acts on the one-dimensional normal space
$\T_{x'_v}(X_v)/\T_{x'_v}(X'_v)$ via
a character $\chi_v$.
The character $\chi_v$ is nontrivial
since $H'_v$
contains the image of $\lambda_v$ which acts nontrivially 
on $\T_{x'_v}(C x_v\cup C x'_v)/{\T}_{x'_v}(Cx'_v)$, cf.~\cite[\S1.1]{B90}.
We define $H_v$ to be the kernel of $\chi_v$ in $H'_v$.
If $v=0$, we simply set $H_v=H$.
\begin{defi} \label{def:sat}
We call the subgroup $H_v$ for $v\in  N \cap \V$
a {\em Brion subgroup} of $G$.
\end{defi}
Brion subgroups are defined up to $G$-conjugacy.
They have the same dimension as $H$ and
they are spherical by \cite[Proposition 1.2]{B90} 
and the proof of \cite[Proposition~1.3]{B90}.

\begin{prop} \label{pro:sat}
Let $v \in N \cap \V$ be a lattice point
in the valuation cone of 
$G/H$. 
\begin{enumerate}
\item The weight lattice $\X^*(G/H_v)$ of the spherical homogeneous 
space $G/H_v$ is equal to the weight lattice $M$ of $G/H$.
\item The valuation cone of $G/H_v$ 
is equal to $\V+\Q v$.
In particular, $v \in \V^\circ$ if and only if the spherical 
subgroup $H_v$ is horospherical, 
where $\V^\circ$ is the relative interior of $\V$. 
\end{enumerate}
\end{prop}

\begin{proof}
We can certainly assume that $v$ is nonzero, 
since the statement is clear for $v=0$.  

(1)
Denote by $\tilde{X}_v$ the total space of the normal bundle
to the $G$-invariant divisor $X'_v$ in $X_v$.
We use the deformation of $X_v$ to the normal bundle
$ \tilde{X}_v$ and we follow the ideas of the proof of \cite[Proposition 1.2]{B90}.

Consider the product $X_v\times\C$ as a spherical
embedding of $G/H \times \C^*$ endowed with
the natural $G \times \C^*$-action.  It is the toroidal simple spherical embedding
corresponding 
to the two-dimensional cone 
in $N_\Q \oplus \Q$ generated by $(v,0)$ and $(0,1)$. 
Let $$p: \tilde{Y}_v\to X_v\times\C$$
be the blow-up of the codimension-two closed $G \times \C^*$-orbit
$X'_v\times\{0\}$ in $X_v\times\C$. 
The preimage by $p$ of $X_v\times\{0\}$ contains a divisor 
which is isomorphic to $X_v$; 
denote by $\tilde{Z}_v$ its complement. 
Let $\rho:\tilde{Z}_v\to\C$ be the restriction
of the composition ${\rm pr}_2\circ p$ to $\tilde{Z}_v\subset \tilde{Y}_v$,
where pr$_2$ is the projection from $X_v\times\C$ to $\C$.
Then $\rho$ defines a fibration over $\C$ such that
$\rho^{-1}(\C^*)$ identifies with
$X_v\times \C^*$, and the fiber $\rho^{-1}(0)$ is isomorphic to $ \tilde{X}_v$. 
The fibration $\rho$ deforms the spherical
variety $X_v$ to $ \tilde{X}_v$.
We have a $G\times\C^*$ action on $\tilde{Z}_v$ so that
$(\tilde{Z}_v,p^{-1}(x_v,1))$
is a smooth 
embedding of $G\times\C^*/H\times\{1\}$
containing a closed codimension-one orbit isomorphic to $ \tilde{X}_v^0\cong G/H_v.$
Choose a point $\tilde{z}'_v\in \tilde{Z}_v$ in this closed $G\times\C^*$-orbit 
of $\tilde{Z}_v$, 
and let $\tilde{H}'_v$ be its stabilizer in $G\times \C^*$. Up to $G$-conjugacy,
$\tilde{H}'_v$ is the image of $H'_v$ by the map
$$j\colon  H'_v \to G\times \C^*, \quad h \mapsto (h,\chi_v (h)).$$
As $G$-varieties, $G\times\C^*/\tilde{H}'_v$ is isomorphic to $G/H_v$,
where $G$ acts on $G\times \C^*$ by left multiplication on the left factor. 

By the Luna-Vust correspondence \cite[\S\S3.3, 7.5, 8.10]{LV} between 
embeddings of $G\times \C^*$ 
and colored fans 
in the valuation cone of $G\times \C^*/H\times\{1\}$, 
the uncolored fan of the blow-up $\tilde Y_v$ is obtained 
from that of $X_v \times\C$ by adding the half-line generated by 
$(v,0) + (0,1) =(v,1)$ 
in the lattice $N \oplus \Z$. 
Moreover, 
the ray generated by the lattice point 
$(v,1)$ corresponds to 
the elementary embedding of $G\times\C^*/H\times\{1\}$ 
with closed $G\times\C^*$-orbit $ \tilde{X}_v^0\cong G/H_v 
\cong G\times\C^*/\tilde{H}'_v$. 

According to \cite[\S3.6]{BP}, we get that 
$$\X^*(G\times\C^*/\tilde{H}'_v) =(M\oplus \Z) \cap (v,1)^\perp,$$
since $\X^*(G\times \C^*)\cong \X^*(B)\oplus \Z$.
The weight lattice of $G\times\C^*/\tilde{H}'_v$ as a $G$-variety is then
the image of $(M\oplus \Z) \cap (v,1)^\perp$ by the projection map
$$\X^*(B)\oplus \Z \to \X^*(B).$$
This image is nothing but $M$.
Therefore the weight lattice of $G/H_v$ is $M$
since $G/H_v\cong G\times\C^*/\tilde{H}'_v$ as a $G$-variety.

(2) 
It is easily seen (see the proof of \cite[Corollary 3.7]{BP}) that
$$\V(G\times \C^*/H\times\{1\}) \cong \V \oplus \Q(0,1).$$
On the other hand, by \cite[Theorem 3.6]{BP},
$\V(G\times \C^*/ \tilde{H}'_v) \cong \V(G/H_v)$ is the quotient
of $\V(G\times \C^*/H)$ by $\Q (v,1)$.
But the image of $\V\oplus \Q(0,1)$ by the isomorphism
$N_\Q \isomap (N_\Q+\Q(0,1)) /\Q(v,1)$
is $\V+\Q v$,
whence
$\V(G/H_v) \cong \V +\Q v.$ 
In particular,
$v \in \V^\circ$ if and only if $\V(G/H_v)=\V+\Q v= N_\Q$. 
Since the 
equality $\V(G/H_v)=N_\Q$ holds 
if and only if $H_v$ is horospherical by \cite[Corollary 6.2]{K}, 
the statement follows. 
\end{proof}

\begin{ex} \label{exa:1}
In the case where $H$ is horospherical, 
all Brion subgroups associated with $H$ are equal to $H$
up to $G$-conjugacy. 
This can be proved by describing all elementary 
embeddings of $G/H$ as induced from elementary embeddings 
of the torus $C/(C\cap H)$. 
\end{ex}

\section{Algebraic approach to Brion subgroups}  \label{sec:alg}
In this section, we investigate an algebraic approach to Brion subgroups 
and we furnish another proof of Proposition \ref{pro:sat} (1). 
Let us fix a primitive lattice point $v$ in $N \cap \V$.

Let $K:=\C(G/H)$ be the field of rational functions over $G/H$.
The field $K$ is a $B$-module 
and it is the quotient field
of the affine coordinate ring $A:=\C[U]$ of the open $B$-orbit $U \subset  G/H$. 
From the restriction of the valuation $v \,:\, K \to \Z$ to the subring $A \subset K$,  
we define the subring 
\[ A_v' := \{ f \in A \;|\;  v(f) \ge 0 \text{ or } f =0\}.\]
By the local structure theorem for toroidal 
embeddings (\cite[\S3.4]{BP} or \cite[Theorem 29.1]{T}), 
the ring $A_v'$ is isomorphic
to the coordinate ring of a $B$-invariant open subset $U' \subset X_v$ 
containing $U$. Note that $U'$ is the union of $U$ and of the open 
$B$-orbit in $X'_v$. Additionnally, $U'$ is isomorphic
to $(\C^*)^{r-1} \times {\C^{s+1}}$, with $r,s$ as in the introduction. 
Since $U'$ is an affine variety with a nonempty intersection 
with $X_v'$, the vanishing ideal $I_v$ 
in $A_v' =\C[U']$ of the divisor $X_v' \cap U'$  
is generated by one element $f_v$.  
One has
$$I_v:=\{f \in A \;|\; v(f) > 0 \text{ or }f=0\}.$$
Then we set for any $i\in\Z$,
$$I_v^{i}:=\{f\in A \;|\; v(f) \ge i\}.$$
We have  a decreasing filtration
$$\cdots
\supset I_v^{-j}\supset\cdots\supset I_v^{-2}\supset I_v^{-1}\supset A'_v  \supset I_v \supset
I_v^{2} \supset \cdots
\supset I_v^{i}\supset \cdots$$
such that:
$$A=\bigcup_{i\in\Z} I_v^{i}.$$
The group $B$ acts on this filtration, and each $I_v^{i}$ is a free $A'_v$-module of
rank one generated by $f_v^i$. We set
$$\widetilde{A_v'}:=\sum\limits_{i\ge 0} I_v^{i}/ I_v^{i+1}, 
\qquad \widetilde{A_v}:=\sum\limits_{i \in\Z } I_v^{i}/ I_v^{i+1}$$
so that $\widetilde{A_v'}= {\rm gr}\,A_v'$ and  $\widetilde{A_v}= {\rm gr}\,A$
with respect to the above filtrations.
Denote by $\overline{A_v'}$ the quotient ring $A'_v/I_v$, which 
is the affine coordinate ring of $X'_v \cap U'$. 
The rings $\widetilde{A_v'}$, $A_v'$ and $\overline{A_v'}$ are
naturally endowed with a $B$-action.
Note that $\widetilde{A_v'}$ is the affine coordinate ring of the normal bundle to the 
divisor $X_v \cap U'$ in $U'$. 
Moreover, $\widetilde{A_v'}$ is the affine 
coordinate of a $B$-invariant open subset $\widetilde{U}'$ in $\tilde{X}_v$ 
which has a nonempty 
intersection with the closed $G$-orbit in $\tilde{X}_v$. 
As for 
$\widetilde{A_v}$, it is
the affine coordinate ring of the open $B$-orbit $\widetilde{U}$ in
$\widetilde{U}'$.

\begin{prop} \label{pro:alg}
There is a natural $B$-equivariant isomorphism between the rings 
$\widetilde{A_v}$ and $A$ which induces an isomorphism between
the groups of $B$-eigenfunctions in $\widetilde{A_v}$ and in $A$.
In particular, the
lattices of weights of the spherical homogeneous spaces 
 $G/H$ and $G/H_v$ are the same.  
\end{prop}

\begin{proof} 
Given any $m \in \X^*(B)$, the dimension of 
the set $K_m^{(B)}$ 
of $B$-eigenvectors of $K$ associated with $m$  
is at most one for spherical $H$. 
Choose for each $m \in M$ a generator $f_m$ of $K_m^{(B)}$ 
and set 
$$\hat{M}:=\sum\limits_{m \in M} \C  f_m.$$
Then $\hat{M}$ is a $\C$-vector subspace of $K$. 
Similarly, let $K_v:=\C(G/H_v)$,  
let 
$M_v$ be the weight lattice of $G/H_v$ and set 
$$\hat{M}_v: =\sum\limits_{m \in M_v } \C  f_{v,m},$$
where for each $m \in M_v$, $f_{v,m}$ is a generator of 
the set 
of $B$-eigenvectors of $K_v$ associated with $m$. 
Let $T \in I_v/I_v^2$ be the class of the generator $f_v \in I_v$. Then we have
\begin{align*}
\widetilde {A_v'} \cong \overline{A_v'}[T]\cong \overline{A_v'} \oplus 
I_v/I_v^2\oplus \cdots\oplus I_v^i/ I_v^{i+1}\oplus\cdots
\qquad \qquad \\
 \subset  \overline{A_v'}[T,T^{-1}]\cong \sum\limits_{i\in\Z}I_v^{i}/ I_v^{i+1} = \widetilde{A_v}.
\end{align*}
For each $i\in\Z$, write $\mathscr{M}_v^{i}$ 
for the image of the projection map
$I_v^{i}\cap \hat{M} \longrightarrow I_v^{i}/I_v^{i+1},$ 
and set 
$$\mathscr{M}_v:= \sum\limits_{i\in \Z}\mathscr{M}_v^{i}.$$

Let us first show that $\mathscr{M}_v=\hat{M}_v$. 
Observing that the intersection $I_v^{i}\cap \hat{M}$ 
is the direct sum of the lines generated by $f_m$, 
where $m \in M$ and $\langle m,v\rangle \ge i$, 
we deduce that $\mathscr{M}_v^{i}$ 
is the direct sum of the lines generated by $f_m$, 
where $m \in M$ and $\langle m,v\rangle = i$.  
By construction, $\mathscr{M}_v$
is contained in $K_v$ and consists of $B$-eigenvectors,
whence a first inclusion 
$\mathscr{M}_v \subset  \hat{M}_v.$ 

To establish the converse inclusion, let us
prove that for each $m\in M_v$, we have {$f_{v,m} \in\mathscr{M}_v$}.
Note that $K_v$ admits a $\C^*$-action induced from the $\Z_{\ge 0}$-grading
on $A_v$ which commutes with the $G$-action.
Then for every $m \in M_v$, the $B$-eigenvector $f_{v,m}$ is also
an eigenvector for the $\C^*$-action, which means that $f_{v,m}$
is homogeneous with respect to the $\Z$-grading on $K_v$.
Hence, $f_{v,m}$ is in 
$\mathscr{M}_v^{i}$ for some $i$.

In conclusion,
$$\hat{M}_v =\mathscr{M}_v \cong \sum\limits_{i\in \Z}
\sum\limits_{m \in M,\atop
\langle m,v\rangle = i} \C f_m \cong \hat{M}.$$
So the weight lattice in $K_v$ is $M$. 
Because the above isomorphisms are $B$-equivariant, the proposition
follows. 
\end{proof}

\begin{rem}
The last proposition can also be deduced from the local structure theorem:
$\Spec A \cong P^{u} \times Z$,
where $Z$ is an elementary embedding of $C/C \cap H$.
\end{rem}

\section{Homogeneous spherical data associated with satellites} \label{sec:Luna}
The idea of classifying spherical
homogeneous spaces $G/H$ in combinatorial terms was proposed by Luna in 2001 \cite{Luna01}. 
For this purpose, Luna invented in \cite[Sect.~3-4]{Luna01} the {\em spherical systems} and 
{\em Luna diagrams}, and conjectured that these combinatorial data can be used to 
classify spherical homogeneous spaces. 
This classification is carried out in two steps. The first one is to reduce the classification to the case 
of {\em wonderful subgroups}, that is, the subgroups $H\subset G$ 
such that $G/H$
admits a wonderful completion. The second
step  is to describe all the wonderful $G$-varieties. 

Luna's conjecture was recently solved due to the 
long efforts of several
researchers. The uniqueness part of Luna's conjecture was proved by Losev 
in 2009 \cite{Losev2009}. 
The existence part was recently completed in a series of papers 
by Bravi and Pezzini \cite{Bravi-Pezzini2014,Bravi-Pezzini2015,Bravi-Pezzini2016}. 
Another proof of the existence part, with different methods, has been proposed by
Cupit-Foutou \cite{Cupit-Foutou}.
We refer the reader to the introduction of \cite{Avdeev2015}, 
and the references given there, for more about this topic. 

In this section we describe the homogeneous spherical data associated 
with Brion subgroups 
in order to prove Theorem \ref{Th:inv-intro}. 
We start by recalling the definition and basic properties 
of the homogeneous spherical datum 
associated with the spherical
homogeneous space $G/H$. 

The Borel subgroup $B$ acts on $G/H$ by left multiplication: 
$(b,gH) \mapsto (bgH)$, 
and the spherical subgroup $H$ acts on the generalized flag variety $B\backslash G$ 
by $(h,Bg) \mapsto Bgh^{-1}$. 
Then we have a natural bijection between the finite cosets 
$B\backslash (G / H)$ and $(B\backslash G) / H$ 
which induces  
a natural bijection $D_i \leftrightarrow D_i'$  
sending $D_i$ to $D'_i:=B \backslash D_i H$ 
between the set  ${\mathscr D}$ of $B$-invariant divisors in $G/H$ and the set
${\mathscr D'}:=\{D_1', \ldots, D_m'\}$ of $H$-invariant divisors
in $B\backslash  G$. 
The Picard group of the generalized flag variety 
$B\backslash G$ is a free group whose base consists of the classes of
line bundles $L_1, \ldots , L_n$
such that the space of global sections
$H^0( B\backslash G , L_j)$ is the $j$-th fundamental representation
(with the highest weight $\varpi_j$) of the semi-simple part of the 
Lie algebra of $G$ 
$(1 \le j \le n)$.  Therefore we can associate with every color
$D_i \subset G/H$
a nonnegative linear combination of fundamental weights:
$\sum_{j=1}^n  a_{ij} \varpi_j = [D_i'] \in {\rm Pic}(B\backslash G)$. 
Let $A := (a_{ij})_{1\le i \le m \atop 1 \le j\le n}$ and let $C:=(c_1, \dots, c_n)$ be the
sum of rows of $A$, i.e., $c_j = \sum_{i=1}^m a_{ij}$ and
\[ \sum_{i=1}^m [D_i'] = \sum_{j=1}^n c_j \varpi_j.\]
It is known (see e.g.~\cite[Lemma 6.4.2]{Luna01} or \cite[Lemma 30.24]{T}) 
that all entries $c_j$ of the vector $C$
belong to the set $\{ 0,1,2 \}$
(so the same statement holds for all 
entries of the matrix~$A$). The matrix $A$ satisfies certain
stronger additional conditions that are used in the classification of colors 
(\cite[Sect.~2 and 3]{Luna97}, \cite[Sect.~1 and 2]{Luna01}). 
Fix $i \in \{1,\ldots,m\}$ and let $J(i)$ be the set of $j\in \{1,\ldots,n\}$ 
such that $a_{ij}\not=0$. We have:
\begin{enumerate}
\item
if $a_{ij}=2$ for some $j$, then
$a_{il} = 0$ for all $l \neq j$, i.e., $[D_i'] = 2 \varpi_j$.
In this case the $B$-invariant divisor $D_i \subset G/H$ is said
to be {\em of type $2a$}, 
\item
if the color $D_i$ is not of type $2a$, then 
\[ [D_i'] = \sum_{j \in J(i)} \varpi_j,\]
and either $c_{j} = 2$ for all $j \in J(i)$ (and the color $D_i$ is called {\em of type $a$}),
or $c_{j} = 1$ for all $j \in J(i)$  (and the color $D_i$ is called {\em of type $b$}).
\end{enumerate}
Denote by ${\mathscr D}^{2 a}$, ${\mathscr D}^a$ and ${\mathscr D}^b$
the set of colors of type $2a$, type $a$ and type $b$, respectively. 

Let us consider two maps
\[  \delta \colon  {\mathscr D} \to {\rm Pic}(B\backslash  G),\qquad \rho \colon
{\mathscr D} \to N, \]
where $\delta(D_i):= [D_i'] \in {\rm Pic}(B\backslash  G)$ $(1 \le i \le m)$
and $\rho(D_i)$ is the lattice point  in $N$ corresponding to
the restriction of the divisorial valuation $\nu_D$ of the field $\C(G/H)$ to the
group of $B$-semiinvariants in $\C(G/H)$. 

The following statement is the uniqueness part of Luna's program:  

\begin{thm}[{\cite[Theorem 1]{Losev2009}}] The triple $(M, \Sigma, {\mathscr D})$ 
consisting of the lattice of weights $M \subset 
{\mathscr X}^*(B)$, the set of spherical roots
$\Sigma \subset M$ and  the set ${\mathscr D}$ of all $B$-invariant
divisors in
$G/H$ together with the two maps $\delta \colon
{\mathscr D} \to {\rm Pic}(B\setminus G)$ and
$\rho \, : \, {\mathscr D} \to N$ uniquely
determines the spherical
subgroup $H \subset  G$ up to conjugation.
\end{thm}

Using the natural bijection between the set ${S}=\{\alpha_1,\ldots,\alpha_n\}$ 
of simple roots of $G$ 
and the set of fundamental weights $\alpha_j \leftrightarrow \varpi_j$
we can regard the set $J(i)$ 
as a subset in $S$. The  set $ {\mathscr D}^a$ of
colors of type $a$ can be characterized 
as the set of those $B$-invariant divisors
$D_i \in {\mathscr D}$ for which the set $J(i)$ contains a spherical root,
i.e., $J(i) \cap
\Sigma \neq \varnothing $. Set $S^p := \{ \alpha \in S\; | \; c_\alpha = 0 \}$,
that is, $S^p$ consists
of those simple roots $\alpha \in S$ such that  the corresponding fundamental
weight is not a summand
of $\delta(D_i)$ for all colors $D_i \in  {\mathscr D}$.
By \cite[\S2.3]{Luna01}, the triple $(M, \Sigma, {\mathscr D})$ 
is uniquely determined by
the quadruple $(M, \Sigma, S^p, {\mathscr D}^a)$, 
where the set ${\mathscr D}^a$ of colors of type
$a$ is equipped with only one map $\rho^{a} \colon {\mathscr D}^a \to N$ 
which is the restriction to ${\mathscr D}^a$ of $\rho$.
The quadruple $(M, \Sigma, S^p, {\mathscr D}^a)$ 
is called the {\em homogeneous spherical datum}
associated with the spherical homogeneous space $G/H$. 

By the Luna general classification, 
the map $G/H \mapsto (M, \Sigma, S^p, {\mathscr D}^a)$ 
is a bijection between spherical
homogeneous spaces of $G$ (up to $G$-equivariant isomorphism) and
homogeneous spherical data for $G$. 

The following theorem has been proved by Gagliardi and 
Hofscheier.

\begin{thm}[{\cite[Theorem~1.1]{GH}}] 
Let $X$  
be a simple toroidal spherical embedding of $G/H$ 
corresponding to an uncolored 
cone $\sigma$. 
Then the homogeneous spherical
datum of the unique closed $G$-orbit in $X$ 
(it is a spherical homogeneous space) is the quadruple
\[ (M_0,  \Sigma_0, S^p,
{\mathscr D}_0^a), \] with
 \[ M_0:= M \cap {\sigma}^{\perp}, \quad \Sigma_0 := \Sigma
 \cap {\sigma}^{\perp},   \quad
 {\mathscr D}_0^a = \{ D_i \in {\mathscr D}^a \;
 | \; J(i) \cap  \Sigma_0 \neq \varnothing \},\]
 where the map $\rho_0 \colon {\mathscr D}_0^a \to N/\langle   {\sigma}
 \rangle$ is 
 the restriction to $ {\mathscr D}_0^a  $ of the map 
 \hbox{$\rho\colon {\mathscr D}^a \to N$} composed
 with the natural homomorphism $N \to  N/\langle   {\sigma} \rangle$.
\end{thm}

\begin{cor} \label{colors-data}
Let $v \in N \cap \V$ be a nonzero lattice point in the valuation
cone  of $G/H$. 
Then the homogeneous spherical
datum of the closed divisorial $G$-orbit $X_v'$ in 
the elementary spherical embedding $X_v$ of $G/H$ corresponding to $v$
is the quadruple
\[ (M_0,  \Sigma_0, S^p,
{\mathscr D}_0^a), \] with
 \[ M_0:= M \cap v^{\perp}, \quad \Sigma_0 := \Sigma \cap v^{\perp},  \quad
 {\mathscr D}_0^a = \{ D_i \in {\mathscr D}^a \; | \; J(i) \cap  \Sigma_0 \neq \emptyset \},\]
  where the map $\rho_0\colon {\mathscr D}_0^a \to N/\langle  v \rangle$ is 
 the restriction to $ {\mathscr D}_0^a  $ of the map  $\rho^{a} \colon {\mathscr D}^a \to N$
 composed
 with the natural homomorphism $N \to  N/\langle v \rangle$.
\end{cor}

We now compare the homogeneous spherical datum     
of the spherical homogeneous space $G/H$ to that 
of the spherical homogeneous
space $G/H_v$ corresponding to the Brion
subgroup $H_v \subset G$.

\begin{thm} \label{Th:inv-data}
Let $v \in N \cap \V$ be a nonzero lattice point
in the valuation cone  of $G/H$. 
Then the homogeneous spherical datum of
the spherical homogeneous space $G/H_v$ is the quadruple
\[ (M, \Sigma \cap v^\perp, S^p, {\mathscr D}^a_v), \]
where ${\mathscr D}_v^a = \{ D_i \in {\mathscr D}^a \; | \; J(i)
\cap  (\Sigma \cap v^\perp) \neq \varnothing  \}$ and the map
$\rho_v^{a} \colon{\mathscr D}_v^a \to N$ is the restriction of $\rho^{a} \colon
{\mathscr D}^a \to N$ to
the subset ${\mathscr D}_v^a \subset {\mathscr D}^a$.
In particular, the $G$-variety $G/H_v$ (and hence the conjugacy class of the Brion subgroup 
$H_v$) only depends on the minimal face of the valuation 
cone of $G/H$ containing the lattice point $v$.
\end{thm}

\begin{proof}
We keep the notations of the proof of
Proposition~\ref{pro:sat}. 
Thus $X'_v$ is the spherical homogeneous space $G/H'_v$, 
where $H_v'$ is the stabilizer of a point in $X'_v$,  
and $H_v$ is contained in $H'_v$. 
Let us consider the spherical homogeneous space $G /H  \times \C^*$ 
together with the $G \times \C^*$-action. Its
lattice of weights is equal to $M\oplus \Z$. 
Then $G/H_v \cong (G \times \C^*)/H_v'$ is the closed $G \times \C^*$-orbit 
in the elementary spherical embedding of
$G/H \times \C^*$ corresponding
to the primitive lattice vector $(v, 1) \in N \oplus \Z$. 
By the definition of colors as 
$B$-stable prime divisors, 
we get a natural
bijection between the set of colors in $G/H \times \C^*$ and those in $G/H$, 
as well as a natural bijection between the set of colors in $G/H_v'$ and 
those in $(G \times \C^*)/H_v'$ that both preserve the type of colors. On the other hand, by 
Corollary~\ref{colors-data}, we have a natural bijection between the set of
$a$-colors $D_i \times \C^*$ in $G/H \times \C^*$ such that
$J(i) \cap (\Sigma 
\cap v^\perp) \neq \varnothing$ and the set
of $a$-colors in $G/H_v \cong (G \times \C^*)/H_v'$. 
Since the composition of the natural embedding $N \hookrightarrow N \oplus \Z$ 
and the epimorphism $ N \oplus \Z \to
(N \oplus \Z)/\langle (v,1) \rangle \cong N$ is the identity
map on $N$ we obtain that the $\rho_v^{a}$-images in $N$ of the
$a$-colors in $G/H_v$ and the $\rho^{a}$-images of the
$a$-colors $D_i$ in $G/H$ such that $J(i) \cap 
(\Sigma  \cap v^\perp) \neq \emptyset$ are the same, i.e., the map
$\rho_v ^{a}\colon{\mathscr D}_v^a \to N$ is the restriction of $\rho^{a} \colon
{\mathscr D}^a \to N$ to
the subset ${\mathscr D}_v^a \subset {\mathscr D}^a$.
\end{proof}

According to Theorem \ref{Th:inv-data}, the following definition is legitimate:

\begin{defi} \label{def2:sat}
Given a subset $I $ of $\Sigma$, the {\em spherical satellite $H_I$ 
of $H$ associated with $I$} (or {\em with the face $\V_I$}) is 
the spherical subgroup $H_v$, where $v$ is any point in the
interior of
$\V_I=\{n \in N_\Q  \; |\; \langle s_i ,n\rangle =0, \text{ for all }s_i \in I\}$.
Then $H_I$ is well-defined up to $G$-conjugation. 
\end{defi}

The spherical satellite $H_\Sigma$
corresponding to the minimal face of $\V$ is $G$-conjugate to $H$.
On the opposite side, there is a unique, up to a $G$-conjugation, horospherical satellite
$H_\varnothing$ which corresponds to the whole cone $\V$
(cf.~Proposition \ref{pro:sat} (2)).
Recall that the valuation cone $\V$ is cosimplicial
and that it has $2^k$ faces.
Consequently, $H$ has exactly $2^k$ spherical satellites.

\begin{rem} 
When $H_v$ is horospherical, that is, $v \in N\cap\V^\circ$,  
Theorem \ref{Th:inv-data} can 
be easily proved by the arguments of Example~\ref{exa:1}.
\end{rem}

As a consequence of Theorem \ref{Th:inv-data} 
we derive the following description of the homogeneous spherical 
datum of the spherical homogeneous space $G/H_I$
corresponding to a satellite $H_I$. 

\begin{cor} 
Given a subset $I $ of $\Sigma$, 
the homogeneous spherical datum of
the spherical homogeneous space $G/H_I$ is the quadruple
\[ (M, I, S^p, {\mathscr D}^a_I), \]
where ${\mathscr D}_I^a = \{ D_i \in {\mathscr D}^a \; | \; J(i)
\cap  I \neq \varnothing  \}$ and the map
$\rho_I^{a} \colon{\mathscr D}_I^a \to N$ is the restriction of $\rho^{a} \colon
{\mathscr D}^a \to N$ to
the subset ${\mathscr D}_I^a \subset {\mathscr D}^a$.
\end{cor}

\section{Normalizers of satellites and consequences} \label{sec:norm} 
Let $N_G(H)$ be the normalizer of $H$ in $G$. 
The homogeneous space $G/N_G(H)$ is spherical
and $N_G(H)$ is of finite index in 
$N_G(N_G(H))$\footnote{\label{foot:counter_example} In \cite[\S5]{BP} or~\cite[Lemma 30.2]{T}, 
it is claimed that $N_G(N_G(H))=N_G(H)$ but 
it is not true in general. 
The simplest counter-example is the following, due to Avdeev. 
Let $G=GL_2$ and let $H$ be the set of diagonal matrices whose 
second entry is 1. Then $N_G(H)$ is the set of diagonal 
matrices, which is not self-normalizing (see \cite[Example 4]{Avdeev2013} 
for another counter-example with $G=SL_3$).}. 
Therefore its valuation cone $\widehat{\V}:=\V(G/N_G(H))$ is strictly convex 
by \cite[Corollary 5.3]{BP}.
Namely, we have $\widehat{\V}=\V/ (\V \cap -\V)$
(see e.g.~the proof of \cite[Theorem 29.1]{T}).
{Let $\widehat{X}$ be the unique, up to a $G$-equivariant isomorphism,
complete simple toroidal embedding of $G/N_G(H)$. 
The corresponding uncolored cone of $N_\Q$ is $\widehat{\V}$. 
The $G$-orbits of $\widehat{X}$ are in bijection with the
faces of $\widehat{\V}$.
Thus we have a natural bijection $I \leftrightarrow \widehat{X}_I$ 
between the subsets $I \subset \Sigma$
and the $G$-orbits $\widehat{X}_I \subset \widehat{X}$ 
such that $\widehat{X}_\Sigma \cong G/N_G(H)$. 

\begin{prop} \label{Pro:norm} 
Let $I$ be a subset of $\Sigma$. 
The normalizer of the stabilizer of any point 
in the $G$-orbit $\widehat{X}_I$ is equal to the normalizer 
$N_G(H_I)$ of the satellite~$H_I$. 
\end{prop}

\begin{proof}
Pick a nonzero point $v$ in the
interior $\V_I^\circ$ of $\V_I$. 
According to \cite[Theorem 15.10]{T}, the canonical map
$G/H\to G/N_G(H)$ extends to
a $G$-equivariant map 
$\pi \colon X_v\to \widehat{X}.$ 
Choose a point $x'_v$ in $X'_v$ 
and set $\widehat{x}_v:=\pi(x'_v).$ 
Then $\widehat{x}_v$ belongs to the $G$-orbit
of $\widehat{X}$ corresponding to the face $\widehat{\V}_I$ of $\widehat{\V}$, 
where $\widehat{\V}_I$  
is the image of $\V_I$  
by the projection map
$N_\Q\to N_\Q/(\V \cap -\V)$. 
Furthermore, since $\pi$ is $G$-equivariant, the stabilizer $H'_v$ of $x'_v$ 
in $G$ 
is contained in the stabilizer $G_{\widehat{x}_v}$ of $\widehat{x}_v$ in $G$.
Hence we have the following inclusions of spherical subgroups:
\begin{align}\label{eq:face}
H_I=H_v \subset   H'_v \subset  G_{\widehat{x}_v}.
\end{align}

Let us show that $G_{\widehat{x}_v}$ is contained in $N_G(H'_v)$.
We first observe that
the fiber of $\pi$ at $x'_v$ is
\begin{align}\label{eq2:face}
\pi^{-1}(\pi(x'_v))= (G_{\widehat{x}_v} ).x'_v \cong
G_{{\widehat{x}_v}}/H'_v.
\end{align}
Since $BN_G(H)=BH$ by \cite[\S5]{BP}, 
the set $BN_G(H)$ is dense in $G$. 
In addition, the parabolic subgroup 
$P$ is the set of $s\in G$ such that $sBN_G(H)=BN_G(H)$. 
Fix an $H$-adapted Levi
subgroup $L$ of $P$, cf.~Definition \ref{def:adapted}.  

Let $\D$ and $\widehat{\D}$ be the set of colors of $G/H$
and $G/N_G(H)$, respectively.
Set
$${X}_v^\circ:=X_v \setminus\bigcup_{D\in\D} \overline{D},
\qquad {\widehat{X}^\circ}:=\widehat{X} \setminus
\bigcup_{D\in\widehat{\D}} \overline{D},$$
and apply the local structure theorem 
for toroidal embddings (cf.~\cite[\S29.1]{T})  to $\widehat{X}$.
There is a closed 
$L$-stable subvariety $\widehat{Z}$ of ${\widehat{X}}^\circ$
on which $(L,L)$ acts trivially and such that the map
$$P^{u} \times \widehat{Z} \longrightarrow {\widehat{X}^\circ}$$
is an isomorphism. Moreover, $\widehat{Z}$ is an embedding 
of the torus $C/C\cap H$ and every $G$-orbits of $\widehat{X}$ meets 
$\widehat{Z}$ along $L$-orbits. 
Here $C$ is the neutral component of the center of $L$. 
Note that $C \cong L/(L,L)$. 

Since ${X}_v^\circ$ is the preimage of ${\widehat{X}^\circ}$
by $\pi$, the local structure for $X_v$ holds with $Z:=\pi^{-1}(\widehat{Z})$
as an $L$-stable variety. Every $G$-orbit of $X_v$ (resp.~$\widehat{X}$)
meets $Z$ (resp.~$\widehat{Z}$) and, hence, 
we may assume that $x'_v\in Z$ and $\widehat{x}_v\in\widehat{Z}$.

Using the isomorphisms ${{X}_v^\circ}\cong P^{u} \times {Z}$   
and ${\widehat{X}^\circ}\cong P^{u} \times \widehat{Z}$, 
we identify $x_v$ with $(1,x_v)$, ${\widehat{x}_v}$ with $(1,{\widehat{x}_v})$, 
and  
the restriction of $\pi$ to ${X}_v^\circ$ with the map 
$$P^{u} \times Z \longrightarrow P^{u}\times \widehat{Z}, 
\quad (p,y) \longmapsto (p,\pi(y)).$$ 
The closed subvarieties 
$Z$ and $\widehat{Z}$ are pointwise fixed by $(L,L)$. 
Consequently, 
$\pi^{-1}(\pi(x'_v)) \cap {X_v^\circ}$ coincides through the above identifications 
with the algebraic subtorus 
$$S:=(C_{{\widehat{x}_v}}).x'_v \cong 
C_{{\widehat{x}_v}}/C_{x'_v}$$ 
of $C$ which 
is contained in the closed subvariety $Z$. 
The Zariski closure $\overline{S}$ of $S$ in $X_v$ is contained 
in $Z$ since $Z$ is closed. 
On the other hand, 
$$\overline{S}\cong \overline{ \pi^{-1}(\pi(x'_v)) \cap {X}_v^\circ} 
= \pi^{-1}(\pi(x'_v)) 
\cong G_{{\widehat{x}_v}}/H'_v.$$
As a result, $G_{{\widehat{x}_v}}/H'_v\cong \overline{S}=S$.  
So by (\ref{eq2:face}), 
$G_{\widehat{x}_v}$ normalizes $H'_v$. 
On the other hand, we notice that $H'_v$ normalizes $H_v$.
Hence by \cite[Lemma 30.2]{T} and the inclusions~(\ref{eq:face}), we
get\footnote{The fact that $N_G(H_v)=N_G(H'_v)$ also results from \cite[\S5.1]{BP}.} 
the expected equalities: 
\begin{align*} 
N_G(H_I)=N_G(H'_v)=N_G (G_{\widehat{x}_v}).
\end{align*}

What is  left is the case where $v=0$. 
In this case, $I=\Sigma$ 
and $\widehat{\V}_{\Sigma}=\{0\}$ so that
$G_{\widehat{x}_v} \cong N_G(H)$. 
But $H_0=H_{\Sigma}=H$ up to $G$-conjugacy, 
whence the statement. 
\end{proof}

We now present some applications of Proposition~\ref{Pro:norm} to the
description of stabilizers of points in spherical embeddings.

Let $\gamma \in N_G(H)$ and let $f$ be a $B$-eigenvector of $\C(G/H)$
with weight $m \in M$. The map $g \mapsto f(g\gamma)$ is invariant by the right action
of $H$, and it is a $B$-eigenvector of $\C(G/H)$
with weight $m$. Therefore for some nonzero complex number $\omega_{H,f}(\gamma)$,
$$f(g \gamma) = \omega_{H,f}(\gamma) f(g),\quad\text{ for all }\quad g \in G.$$
The map
$$\omega_{H,f} \colon N_G(H) \rightarrow \C^*,\quad \gamma \mapsto 
\omega_{H,f}(\gamma),$$
is a character whose restriction to $H$ is trivial, and which only depends on $m$,
\cite[\S4.3]{Br97}.
We denote it by $\omega_{H,m}$. 
The reader is referred to \cite[Theorem 21.5]{T} for a proof of the following result. 
\begin{lemma} 
\label{lem:R}
The spherical subgroup $H$ is the kernel of the homomorphism
$$N_G(H) \rightarrow \Hom(M, \C^*), 
\quad \gamma \mapsto (m\mapsto
\omega_{H,m}(\gamma)).$$ 
\end{lemma}

Fix now an arbitrary simple embedding $X$ of $G/H$, and choose a point $x'$ 
in the
unique closed $G$-orbit of $X$.
There exist a simple toroidal $G/H$ embedding $\tilde{X}$ and a
proper birational $G$-equivariant
morphism $\pi :\tilde{X} \to X$ such that $\pi(\tilde{x}')=x'$
for some $\tilde{x}'$ in the unique closed $G$-orbit of $\tilde{X}$ 
(see e.g.~\cite[Proposition 2]{Br97}).
Let $(\sigma,\F)$ be the colored cone corresponding to the
simple embedding $X$, and
let $I(\sigma)$ be the set of all spherical roots in
$\Sigma$ that vanish on $\sigma$. 
Thus $\V_{I(\sigma)}$ is the minimal face of $\V$ containing $\sigma$.
Note that $\sigma \cap \V$ is the uncolored cone corresponding to the
simple toroidal embedding $\tilde{X}$.

We are now in a position to prove Theorem \ref{Th:stab}. 
\begin{proof}[Proof of Theorem \ref{Th:stab}] 
By choosing $X= \tilde X$,  we may actually assume that 
$X$ is toroidal. 
Let $v$ be a nonzero element in $\sigma \subset \V_{I(\sigma)}$. 
By \cite[Theorem 15.10]{T}, 
the identity map 
$G/H \to G/H$ and 
the canonical map
$G/H \to G/N_G(H)$ extend to
$G$-equivariant maps 
$$\pi_v \colon X_v \to {X} \quad \text{ and }\quad  
\pi \colon X \to \widehat{X},$$
respectively. 
As in the proof of Proposition \ref{Pro:norm}, 
we deduce from this the inclusions 
of spherical subgroups, 
\begin{align} \label{eq:stab2} 
H_{I(\sigma)}=H_v \subset H'_v \subset G_{x'} 
\subset G_{{\widehat{x}}},
\end{align}
where $H'_v$ is the stabilizer in $G$ of a point $x'_v$ 
in the closed $G$-orbit $X'_v$ 
such that $x'=\pi_x(x'_v)$ 
and $\widehat{x}:=\pi(x')$. 
By Proposition \ref{Pro:norm}, we have 
$N_G(H_{I(\sigma)}) = N_G(G_{{\widehat{x}}}).$ 
So by \cite[Lemma 30.2]{T} and (\ref{eq:stab2}), 
it follows that 
$N_G(G_{x'})=N_G(H_{I(\sigma)}).$ 
Then Lemma~\ref{lem:R} applied to the spherical
subgroup $G_{x'}$ implies that 
$$G_{x'} = \{ \gamma \in N_G(H_{I(\sigma)})\;|\; \omega_{G_{x'},m}(\gamma)=1, 
\text{ for all } m \in M\cap \sigma^\perp  \},$$ 
since the weight lattice of $G/G_{x'}$ is $M\cap \sigma^\perp$. 
This finishes the proof of the theorem.  
\end{proof}

Return to the case where $X$ is a simple embedding of $G/H$ (not necessarily toroidal). 
As a consequence of Theorem \ref{Th:stab}, we get the following 
result announced in the introduction. 

\begin{cor} \label{cor:stab}
With the preceding notation, the stabilizer $G_{x'}$ of $x'$ in $G$ contains the satellite $H_{I(\sigma)}$.
\end{cor}

\section{Limits of stabilizers of points in arc spaces} \label{sec:arc}
Let $\K:=\C((t))$ be the field of formal Laurent series, 
and let $\O:=\C[[t]]$ be the ring of formal power series.
For $X$ a scheme 
over $\C$,
denote by $X(\K)$ and $X(\O)$ the sets of $\K$-valued points
and $\O$-valued points of $X$, respectively.
If $X$ is a 
variety admitting an action of an algebraic group $A$,
then $X(\K)$ and $X(\O)$ both admit a canonical action of the group $A(\O)$
induced from the $A$-action on $X$.
For example, the action of the group $G$ on the homogeneous 
space $G/H$ by left multiplication yields 
an action of the group $G(\O)$ on $(G/H)(\K)$.

As shown by Luna and Vust, the set $N \cap \V$ 
parameterizes the $G(\O)$-orbits in $(G/H)(\K)$. 
More precisely, to each $\lambda \in(G/H)(\K)\setminus (G/H)(\O)$ one associates a valuation 
$v_\lambda\colon \C(G/H)^*\to \Z$ as follows (cf.~\cite[\S4.2]{LV}). 
The action 
of $G$ on $G/H$ induces a dominant morphism 
$$G\times \Spec \K \to G\times G/H \to G/H,$$
where the first map is $1\times \lambda$, whence an injection of fields 
$\iota_\lambda \colon \C(G/H)\to \C(G)((t)).$ 
Then we set  
$$v_\lambda:=v_t\circ \iota_\lambda \colon \C(G/H)^*\to \Z,$$ 
where $v_t:\C(G)((t))^*\to \Z$ is the natural discrete valuation on $\C(G)((t))$ given 
by the order in $t$ of the lowest term of formal series.
If $\lambda\in (G/H)(\O)$, we define $v_\lambda$ 
to be the trivial valuation. By \cite[\S\S4.4, 4.6]{LV}, 
$v_\lambda$ is a $G$-invariant discrete 
valuation of $G/H$, that is, an element of $N\cap \V$. 
Furthermore,  
for every $\gamma\in G(\O)$ and every $\lambda\in (G/H)(\K)$, 
we have $v_{\gamma \lambda}=v_{\lambda}.$ 

By \cite[Proposition 4.10]{LV} (see also 
\cite[Theorem 3.2.1]{GN}), 
the mapping 
\begin{align}\label{eq:mapping}
(G/H)(\K) \rightarrow N\cap \V,\quad \lambda 
\mapsto v_\lambda, 
\end{align} is a surjective map and the fiber over any $v\in
N\cap \V$ is precisely one $G(\O)$-orbit. Thus there is a one-to-one 
correspondence between the set of $G(\O)$-orbits in $(G/H)(\K)$ and the 
set $N\cap \V$. 
We write $\CC_v:=\{\lambda\in (G/H)(\K)\;|\; v_\lambda=v\}$ 
for the $G(\O)$-orbit 
in $(G/H)(\K)$ corresponding to $v\in N\cap \V$. 

Let $X$ be a spherical embedding of $G/H$ associated 
with a colored fan $\mathbb{F}$. 
The {\em support} of 
$\mathbb{F}$ is the set 
$$|\mathbb{F}| = 
\bigcup_{({\mathcal C},\F)\in \mathbb{F}} {\mathcal C}
\cap \V.$$ 
The group $G(\O)$ 
acts on the sets $X(\O)$ and on $(G/H)(\K)$ which are both 
viewed as subsets of $X(\K)$. 
Hence the group $G(\O)$ acts on 
$X(\O)\cap (G/H)(\K)$. 
By restriction to $X(\O)\cap(G/H)(\K)$, the mapping \eqref{eq:mapping}
induces a surjective mapping  
$$X(\O)\cap(G/H)(\K) \rightarrow 
N \cap |\mathbb{F}|,\quad \lambda \mapsto v_\lambda,$$
whose fiber over any $v\in 
N \cap |\mathbb{F}|$ is precisely one $G(\O)$-orbit. Thus there is a 
one-to-one correspondence between the set of $G(\O)$-orbits in $X(\O)\cap 
(G/H)(\K)$ and the set $N \cap |\mathbb{F}|$. 

Fix an $H$-adapted Levi subgroup
$L$ of $P$, 
and denote as usual by $C$ the neutral component of its center.  
Let $x:=eH$ be the base point in $G/H$. Any one-parameter subgroup 
$\lambda$ of $C$ yields a morphism of algebraic varieties, 
$$\widetilde{\lambda}\colon \C^*\to G/H,\quad t \mapsto \lambda(t) x, $$ 
whose comorphism gives rise to an element $\widehat{\lambda}$ 
of $(G/H)(\K)$. 

Pick a nonzero lattice point $v$ in $N\cap \V$. 
Since the open orbit of $X_v$ is $G/H$, one can assume 
that $x_v=eH$. Then, $\lambda$ is adapted to the 
elementary embedding $(X_v,x_v)$ (cf.~Definition~\ref{def:param}) 
if and only if $\widehat{\lambda}$ 
lies in $X(\O)\cap (G/H)(\K)$ 
and $\lim_{t\to 0} \widehat{\lambda}(t)$ 
exists in the closed $G$-orbit $X'_v$.  
One can choose $\lambda_v \in \X_*(C)$, adapted to $(X_v,x_v)$, such that 
$v_{\lambda_v}=v$.  
By \cite[\S5.4]{LV}, 
$$G(\O) \widehat{\lambda}_v=\CC_v,$$
and so $v=v_{\lambda_v}=v_{\widehat{\lambda}_v}$.  
Setting $\widetilde{\lambda}_v(0):=\lim_{t\to 0} \widetilde{\lambda}_v(t)$, 
one gets a morphism of algebraic varieties 
$$\widetilde{\lambda}_v \colon \C \longrightarrow X_v.$$
One can view the differential $\d \widetilde{\lambda}_v(0)$ of 
$\widetilde{\lambda}_v$ at $0$ as an element of $T_{x'_v}(X_v)$, 
where $x'_v:=\lim_{t\to 0}\lambda_v(t) x_v \in X'_v$. 

\begin{lemma} \label{Lem:tangent} 
The one-dimensional space $\T_{x'_v}(X_v)/\T_{x'_v}(X'_v)$ is
generated by the class of $\d \widetilde{\lambda}_v(0)$ 
in $\T_{x'_v}(X_v)/\T_{x'_v}(X'_v)$.
\end{lemma} 

\begin{proof} 
According to the local structure theorem for toroidal 
embeddings (\cite[\S3.4]{BP} or \cite[Theorem 29.1]{T}), it suffices 
to show the statement for elementary embeddings of a torus. 
So we can assume that $X_v = \mathbb{T} \cup X'_v$ is an elementary 
embedding of the algebraic torus $ \mathbb{T} :=C/C \cap H$ in $\C^n$. 
Arguing in coordinates $x_1,\ldots,x_n$, we can assume furthermore 
that the divisor $X'_v$ 
is given by the equation $x_1=0$. 
One can choose for $\lambda_v$ 
the adapted one-parameter subgroup $t \mapsto (t,1,\ldots,1)$. 
The statement is now clear. 
\end{proof}

\begin{lemma} \label{lem:CH}
We have $C\cap H_v=C\cap H.$
\end{lemma}

\begin{proof} 
It follows from the definition of the Brion subgroup $H_v$ (cf.~Definition \ref{def:sat}) 
that $C \cap H_v$ is a Brion subgroup associated 
with the elementary embedding $C x_v \cup C x'_v$  
of the algebraic torus $ \mathbb{T}=C/C \cap H$. 
Since $\mathbb{T}$ is a torus,  
$C \cap H$ is equal to each of the Brion
subgroups of $\mathbb{T}$. Therefore we get $C\cap H$ = $C \cap H_v$.  
\end{proof}

Theorem \ref{Th:arc} can now be proved.
 
\begin{proof}[Proof of Theorem \ref{Th:arc}] 
We have to show that 
$H_v=\{\lim_{t\to 0}\gamma(t) 
\, ; \,\gamma\in G(\O)_{{\widehat{\lambda}_v}}\},$ 
where $G(\O)_{{\widehat{\lambda}_v}}$ is the stabilizer 
of $\widehat{\lambda}_v$ in $G(\O)$, that is, 
the set of $\gamma \in G(\O)$ such that $\gamma(t)
(\lambda_v(t)x_v)=\lambda_v(t)x_v$ for all $t\in\C^*$.  

In order to establish the inclusion $H_v \supset \{\lim_{t\to 0}\gamma(t) 
\, ; \,\gamma\in G(\O)_{{\widehat{\lambda}_v}}\}$, 
fix $\gamma\in G(\O)_{{\widehat{\lambda}_v}}$ and write  
$\gamma(0)$ for its limit at $0$. 
Let $H'_v$ be the stabilizer 
of $x'_v$ in $G$. One can 
assume that $H_v =\ker \chi_v \subset H'_v$ (cf.~Definition \ref{def:sat}). 
For any $t\in\C^*$,
\begin{align} \label{eq:arc}
\gamma(t)(\lambda_v(t) x_v)=\lambda_v(t)x_v.
\end{align}
Taking the limit at $0$ on both sides, we get that 
$\gamma(0) x'_v=x'_v$.
Hence $\gamma(0) \in H'_v$. 
Since
$\gamma(0) \in H'_v$, $\gamma(0)$ leaves invariant 
the tangent space ${T}_{x'_v}(X'_v)$.
Differentiating the equality (\ref{eq:arc}), we get
$$\gamma(0) \left(d \widetilde{\lambda}_v(0)\right)= 
d \widetilde{\lambda}_v(0) \mod {T}_{x'_v}(X'_v).$$
Since $H'_v$ acts on the one-dimensional
space ${T}_{x'_v}(X_v)/{T}_{x'_v}(X'_v)$ by the character $\chi_v$,
we deduce from Lemma \ref{Lem:tangent} 
that $\gamma(0)$ is in the kernel $H_v$ of $\chi_v$, 
whence the expected inclusion.

Let us prove the other inclusion. 
Pick $\gamma_0 \in H_v$.
We have to show that $\gamma_0=\lim_{t\to 0} \gamma(t)$
for some $\gamma\in G(\O)_{{\widehat{\lambda}_v}}$.
By \cite[Proposition and Corollary 5.2]{BP}, 
$N_G(H_v) = H_v^0 (C \cap N_G(H_v))$, 
with $H_v^0$ the neutral component of $H_v$, whence 
$ H_v \subset H_v^0 (C \cap H_v).$ 
Write $\gamma_0= \tilde{\gamma}_0 c$
with $\tilde{\gamma}_0 \in H_v^0$ and  $c \in C\cap H_v$.
Since $H_v^0$ is connected, 
there exists $l \in \Z_{>0}$ and 
$\eta_1,\ldots,\eta_l$ in $\mathfrak{h}_v:={\rm Lie}(H_v)$ 
such that 
$$\tilde{\gamma}_0=  \exp_{H_v}(\eta_1)\ldots \exp_{H_v}(\eta_l),$$
where 
$\exp_{H_v} \colon 
\mathfrak{h}_v \to H_v$ is the exponential map of $H_v$.

Denote by $(g,\xi) \mapsto g.\xi$ the adjoint action of 
the algebraic group $G$ on its Lie algebra. 
By \cite[Lemma 1.2]{B90}, $\lim_{t\to 0} \lambda_v(t). \xi$ 
exists and lies in Lie($H'_v$) for every $\xi$ 
in the Lie algebra $\mathfrak{h}$ of $H$. 
Moreover, according to \cite[Proposition 1.2]{B90}\footnote{The result is stated
in \cite[Proposition 1.2]{B90} for $H_v$ equals to its normalizer. 
However the proof
does not use this hypothesis.}, the Lie algebra of $H_v$ is
$$\lim_{t\to 0}\lambda_v(t). \mathfrak{h} 
:= \{ \lim_{t\to 0} \lambda_v(t). \xi \; ; \; \xi \in \mathfrak{h} \}.$$
Therefore for some 
$\xi_1,\ldots,\xi_l \in \mathfrak{h}$, 
$\eta_i=\lim_{t\to 0}\lambda_v(t).\xi_i$.  
Set for any $t\in\C^*$,
$$\tilde{\gamma}(t) :=\lambda_v(t)\exp_{H}(\xi_1) 
\ldots\exp_{H}(\xi_l) \lambda_v(t^{-1}),$$
where $\exp_{H} \colon\mathfrak{h} \to H$ is the exponential map of $H$.
Then for any $t\in \C^*$,
\begin{align*}
\tilde{\gamma}(t)(\lambda_v(t)x_v)
=\lambda_v(t)\exp_{H}(\xi_1) 
\ldots\exp_{H}(\xi_l)x_v
=\lambda_v(t)x_v
\end{align*}
since for $i =1,\ldots,l$, $\exp_H(\xi_i)$ lies in $H=G_{x_v}$.
As a result $\tilde{\gamma} \in G(\O)_{{\widehat{\lambda}_v}}$.

Using the natural embedding $G\hookrightarrow G(\O)$, 
one can view the element $c$ of $C\subset G$ 
 as an element of $G(\O)$. Set for any $t\in\C^*$,
$ \gamma(t):=\tilde{\gamma}(t)c.$ 
We have
\begin{align*}
\lim_{t\to 0}\tilde{\gamma}(t)
&=\lim_{t\to 0} \left(\lambda_v(t)\exp_H(\xi_1)\ldots \exp_H(\xi_l)\lambda_v(t^{-1})
\right) &\\
&=\lim_{t\to 0}\left( 
\exp_{G}(\lambda_v(t).\xi_1) \ldots 
\exp_{G}(\lambda_v(t).\xi_l)\right) 
=\exp_{H_v}(\eta_1) \ldots \exp_{H_v}(\eta_l)
=\tilde{\gamma}_0,
\end{align*}
with $\exp_G$ the exponential map of $G$, 
whence $\lim_{t\to 0}\gamma(t)=\tilde{\gamma}_0 c
= \gamma_0.$ 
By Lemma~\ref{lem:CH}, the element $c$ of $C\cap H_v$ 
is in $C\cap H$. In particular, $c$ is in the stabilizer of $x_v$ in $G$. 
Hence, 
$$\gamma(t)(\lambda_v(t)x_v)=\tilde{\gamma}(t)c(\lambda_v(t)x_v)
=\tilde{\gamma}(t)(\lambda_v(t)cx_v)
=\tilde{\gamma}(t)(\lambda_v(t)x_v)
=\lambda_v(t)x_v$$
because $\tilde{\gamma}\in G(\O)_{{\widehat{\lambda}_v}}$
and $c$ commutes with the image of $\lambda_v$. 
This proves that $\gamma$ is in $G(\O)_{{\widehat{\lambda}_v}}$. 
We have shown that 
$\gamma_0 \in \{\lim_{t\to 0}\gamma(t) 
\;;\;\gamma\in G(\O)_{{\widehat{\lambda}_v}}\}$ 
as desired.
\end{proof}

Theorem \ref{Th:arc} can be used in practice to compute the Brion 
subgroups $H_v$, where $v \in N\cap \V$. 
According to Theorem \ref{Th:inv-intro}, it suffices to compute finitely 
many of them to describe all the satellites $H_I$ of $H$, 
where $I \subset  \Sigma$. 
We illustrate this with some examples. 
In the following, 
there is no distinction made between  
$\lambda$ and ${\widehat{\lambda}}$, and so 
by a slight abuse of notation we write 
$G(\O)_{{\lambda}}$ for $G(\O)_{{\widehat{\lambda}}}$.

\begin{ex} \label{exa3:arc} 
In the setting of~\S\ref{sub:mot}, 
assume that $G=GL_n$ $(n\in \N^*)$ and that $T=D_n$ 
is the set of diagonal matrices in $G$. 
Then $T\times T$ is a maximal torus of $G\times G$ 
and it is adapted to $\Delta(G)$, cf.~\cite[\S2.4, Example~2]{Br97}.
Hence,
$$M \cong\X^*(D_{n}) \cong \Z^n.$$
The dual lattice $N=\Hom(M,\Z)$ identifies with the lattice of one-parameter 
subgroups of $D_{n}$, 
and $\V$ identifies with the set of sequences
$(\lambda_1,\ldots,\lambda_n)\in\Q^n$ with $\lambda_1 \ge\cdots\ge\lambda_n$.
Let $T^{+}_{n}$ (resp.~$T^-_{n}$) be the subgroup of $GL_{n}$
consisting of upper (resp.~lower) triangular matrices. 
The identity matrix $E$ lies in the open $T^+_{n} \times T_{n}^-$-orbit since $T^+_{n} T_{n}^-$
is open in $GL_{n}$. 
Consider the one-parameter subgroup $\lambda$ of $T \times T$ 
given by: 
$$\lambda( t)  = \left(\diag(t^{\lambda_1},\ldots,t^{\lambda_n}),
\diag (t^{-\lambda_1},\ldots,t^{-\lambda_n})\right), \; \text{ for all }\; t\in\C^*,$$
where 
$(\lambda_1,\ldots,\lambda_n)$ in $\Z^n$ and $\lambda_1
\ge\cdots\ge\lambda_n$. 

Write 
$(\lambda_1,\ldots,\lambda_n)=(\mu_1^{k_1},\ldots,\mu_r^{k_r})$ 
with $\mu_1> \cdots > \mu_r$ and $k_1+\cdots+k_r=n$ \hbox{$(k_i \in \N^*$)}. 
Let $P_\lambda$ be the (standard) parabolic subgroup 
of $GL_n$ 
consisting 
of the block upper triangular invertible matrices with diagonal blocks of sizes 
$k_1,\ldots,k_r$. 
Denote by $L_\lambda$ its standard Levi factor, 
and let $P_{-\lambda}$ 
be the opposite parabolic subgroup 
to $P_\lambda$ so that $P_{-\lambda}\cap P_\lambda=L_\lambda$. 

An easy computation 
allows us to determine the set of limits of points $(x(t),y(t)) \in G(\O)_{\lambda}$, 
and applying Theorem \ref{Th:arc}, we get that 
$$
H_{\lambda} = 
(P_{-\lambda}^{u} \times P_\lambda^{u})\Delta (L_\lambda),$$
where $P_\lambda^{u}$ and $P_{-\lambda}^{u}$ 
are the unipotent radical of $P_\lambda$ and $P_{-\lambda}$, 
respectively. 
Thus we recover the description of the satellites of \S\ref{sub:mot} 
for $G=GL_n$. 
\end{ex}

\begin{ex} \label{exa4:arc}
Let $n\in\Z$, $n \ge 3$. Assume that $G=SL_{n}$ 
and that $H\cong SL_{n-1}$ is the subgroup
$H=\left\{\begin{pmatrix} A & 0 \\
0 & 1\end{pmatrix} \; ; \; A \in SL_{n-1} \right\}.$ 
The spherical homogeneous space $G/H$ can be realized as
$$G/H\cong\{(x,y)\in\C^n\times(\C^{n})^* \; |\; \langle y,x\rangle=1\}.$$
The Levi subgroup
$L=\left\lbrace\begin{pmatrix}
a & 0 & 0 \\
0 & B & 0 \\
0 & 0 & c
\end{pmatrix} \, ; \, a,c \in \C^*, B \in GL_{n-2} \, , 
ac \det B=1 \right\rbrace$ 
is adapted to the stabilizer $\tilde{H} \cong H$ of the point 
$((1,0,\ldots,0,1),(1,0,\ldots,0))$.  
Let $\lambda$ be the one-parameter subgroup of 
the neutral component of the center of $L$, 
with $\lambda(t)= {\rm diag}(t,1,\ldots,1,  ,t^{-1}) $ for all $t\in\C^*$.
In order to compute $G(\O)_\lambda$, we exploit the elementary embedding
$$X:=\{[x_1;\ldots;x_n;y_1;\ldots;y_n;z_0]\;|\;\sum x_i y_i=z_0^2\}\subset \P^{2n}$$
of $G/H$.
The closed $G$-orbit of $X$ corresponds to the divisor $\{z_0=0\}$, 
and we have
$$\lambda(t)([1;0;\ldots,0;1;1;0;\ldots;0;1])=
[t;0;\ldots;t^{-1};t^{-1};0;\ldots,0;1]\in \P^{2n}.$$
From this, we easily determine $G(\O)_\lambda$ and 
by Theorem \ref{Th:arc}, we get that 
$$
H_\varnothing \cong 
\left\{\begin{pmatrix} a & 0 & 0 \\
u & B & 0 \\
v & w & a \\
\end{pmatrix}  \; ;\; B \in SL_{n-2}, \,u,v,w,a \in \C, \, a^2=1 \right\}.$$
\end{ex}

\section{Virtual Poincar\'{e} polynomials of $G/H_I$} \label{sec:Poincare}
This section is devoted to the proofs of Theorem \ref{Th2:Poincare} and 
Theorem~\ref{Th:P-one}. 
They 
are strongly inspired by the work of Brion and Peyre \cite{BPe}. 

In order to state the results of \cite{BPe} we need,  
let us assume 
for a while that $G$ is any complex connected linear algebraic group 
and that $H$ is any closed 
subgroup of $G$. 
We denote by $r_H$ the rank of $H$ and by $u_H$ the 
dimension of a maximal unipotent subgroup of $H$. 
Choose maximal reductive 
subgroups $H^{red} \subset  H$ and $G^{red} \subset  G$ such that 
$H^{red} \subset  G^{red}$. 

\begin{thm}[{\cite[Theorem 1(b)]{BPe}}] \label{Th:Brion-Peyre} 
Under the above assumptions, 
there exists a polynomial $Q_{G/H}$ with nonnegative integer coefficients 
such that 
$$\tP_{G/H}(t) = t^{u_G -u_H} (t-1)^{r_G-r_H} Q_{G/H}(t),$$
and
$$Q_{G/H}(t) = Q_{G^{red}/H^{red}}(t).$$
In particular, if $G$ is reductive, then  
$$\tP_{G/H}(t) = t^{u_{H^{red}} - u_H} \tP_{G/H^{red}}(t).$$
Moreover, $Q_{G/H}(0)=1$ if $H$ is connected. 
\end{thm}

Let us now return to the case where $G$ is reductive 
and $H$ is a spherical subgroup of $G$. 
For each subset $I$ of the set of 
spherical roots $\Sigma$ 
of $G/H$, 
choose a maximal reductive subgroup $H_I^{red} \subset  H_I$ of the satellite $H_I$ 
of $H$. We briefly denote by $r_I$, $u_I$ and $u_I^{red}$ the integers $r_{H_I}$, 
$u_{H_I}$ and $u_{H_I^{red}}$, respectively. Note that $r_{I} = r_{H_I^{red}}$. 

\begin{lemma}[{\cite{BPe}}] \label{lem:Brion-Peyre}
Let $I \subset  \Sigma$. 
Then $H_I^{red}$ is contained in a $G$-conjugate of $H$. 
\end{lemma}  

\begin{proof}
If $I =\Sigma$, the statement is clear. 
Let $I \subsetneq \Sigma$, and pick 
a lattice point $v$  in $N\cap \V_I^\circ$. 
Without loss of generality, 
we may assume that $H_I^{red}=H_v^{red}$ is contained in $H_v$. 

We now resume the arguments of \cite[Sect.~2]{BPe}. 
Consider the action of $H_v^{red}$ on the tangent space $T_{x'_v}(X_v)$ 
and choose an $H_v^{red}$-invariant complement $N$ to the 
tangent space $T_{x'_v}(X'_v)$. 
By construction, $H_v^{red}$ fixes $N$ pointwise. 
Then we can find an $H_v^{red}$-invariant subvariety $Z$ 
of $X_v$ such that $Z$ is smooth at $x'_v$ and  
$T_{x'_v}(Z)=N$. Therefore $H_v^{red}$ fixes pointwise a neighborhood 
of $x'_v$ in $Z$ and this neighborhood meets the open $G$-orbit 
$G/H$. Thus we may suppose that $H_v^{red} = H_I^{red}$ is contained in $H$, 
as expected. 
\end{proof}

Since the satellite $H_I$ is defined up to $G$-conjugacy, 
Lemma \ref{lem:Brion-Peyre} allows us to assume from 
now on that $H_I^{red}$ 
is contained in $H$ for each subset $I\subset  \Sigma$. 

\begin{thm} \label{Th2:Poincare}
Assume that $H$ is connected, and let $I \subset \Sigma$.
\begin{enumerate}
\item We have:
$$\tP_{G/H_I}(t) = \tP_{G/H}(t) \tP_{H/H_I^{red}}(t) t^{-(u_I -u_{I}^{red})}.$$
In particular, the ratio 
$$R_I(t^{-1}):=\dfrac{\tP_{G/H_I}(t)}{ \tP_{G/H}(t)}$$ is a polynomial $R_I$ in $t^{-1}$ with integer 
coefficients and constant term 1. 
\item The degree of $R_\varnothing$ is $u_G - u_H=u_\varnothing - u_H$.
\item The degree of $R_I$ is $u_I -u_H$, provided that $H_I$ is connected. 
\end{enumerate}
\end{thm}

\begin{proof}
(1) Set $X=G/H_I^{red}$, $Y=G/H$ and $F=H/H_I^{red}$, 
and consider the locally trivial fibration $F \hookrightarrow 
X \stackrel{f}{\to} Y$ for the complex analytic topology. 

\begin{claim}\label{claim:local-system}
The local systems $R^{j}f_* \C_{X}$ are constant  
for each $j$.
\end{claim}

\begin{proof} 
Recall that 
$Y$ being a connected complex algebraic variety with
the complex analytic topology, 
for any choice of a base point $y \in Y$, there in an equivalence 
of categories between the category of local systems on $Y$ 
and the category of representations of the fundamental group $\pi_1(Y,y)$. 
So for each $j$ there is a natural $\pi_1(Y,y)$-action on 
the cohomology fiber 
$R^{j}f_* (\C_{X})_y \cong H^j(f^{-1}(y),\C) \cong  H^j (F,\C)$. 
Hence it is enough to show that this action is trivial. 
We are dropping, later on, the mention of the base point for the fundamental 
group. 

Since the fiber $F=H/H_I^{red}$ is connected ($H$ being connected), 
we deduce 
from the long exact sequence of homotopy groups, 
$$\cdots \to \pi_1(F) \to \pi_1(X) \to \pi_1(Y) \to \pi_0(F)=1 \to \cdots, $$ 
that the morphism $ \pi_1(X) \to \pi_1(Y)$ is surjective. 

Let us now make use of 
the pullback of the fiber bundle $X \to Y$ by the map $f \colon X \to Y$. 
We have $f^*(X) \cong X \times F$ and the following diagram 
commutes, 
where the below and left arrows are the natural projections: 
\begin{center}
\hspace{0.25cm}\xymatrix{X \ar[r]^{f} & Y  \\ f^*(X) \cong X \times F\ar[u]  \ar[r] & X \ar[u]_{f}} 
\end{center}
The natural action of $\pi_1(X)$ on the cohomology 
fiber $H^j (F,\C)$ 
furnishes a map $\rho \colon  \pi_1(X) \to {\rm Aut}(H^j (F,\C))$
which 
is the composition of the surjective map $ \pi_1(X) \to \pi_1(Y)$ 
and the representation $\pi_1(Y) \to {\rm Aut}(H^j (F,\C))$ 
induced by the local system $R^{j}f_* \C_{X}$. 
Since the fibration $f^*(X) \cong X \times F \to X$ is trivial, 
the action of $\pi_1(X)$ on $H^j (F,\C)$ is trivial, that is, 
$\rho$ is trivial. 
The map $ \pi_1(X) \to \pi_1(Y)$ being surjective, we 
deduce that the representation $\pi_1(Y) \to {\rm Aut}(H^j (F,\C))$ 
is trivial too, as required. 
\end{proof}

By Claim \ref{claim:local-system}, 
we can apply \cite[Theorem 6.1 (ii)]{DL97}  
to the algebraic morphism $f$ 
to obtain:
$$\tP_{G/H_{I}^{red}}(t) = \tP_{G/H}(t) \tP_{H/H_{I}^{red}}(t).$$
Thus by Theorem \ref{Th:Brion-Peyre}, we get
\begin{align*}
\tP_{G/H_{I}}(t) 
= \tP_{G/H}(t) \tP_{H/H_{I}^{red}}(t) t^{ -( u_I-u_I^{red})}. 
\end{align*}
In addition, $\tP_{H/H_{I}^{red}}(t)$ is a polynomial with integer 
coefficients. 
Since $\dim G/H_I =\dim G/H$, 
$\tP_{G/H_I}(t)$ and $\tP_{G/H}(t)$ have the same degree,
and so the degree of $\tP_{H/H_I^{red}}(t)$ must be equal to
$d:=u_{I}-u_{I}^{red}.$ 
Hence for some integers $a_0,\ldots,a_{d-1},a_d $, we have
\begin{align*}
P_{H/H_I^{red}}(t) t^{-(u_{I}-u_{I}^{red})}
&= ( a_0 +a_1 t+ \cdots + a_{d-1} t^{d-1} + a_d t^{d}) t^{-d} 
= R_I(t^{-1}), &
\end{align*}
where 
$R_I(X)= a_{0} X^{d} + a_{1} X^{d-1}+ \cdots + a_{d-1} X + a_d$. 
It remains to show that the constant term $a_d$ of $R_I$ is 1. 
Since $G/H_I$ and $G/H$ are irreducible of the same dimension,
the leading term 
of $\tP_{G/H_I}(t)$ and $\tP_{G/H}(t)$ is $t^{\dim G/H}$,
whence $a_d=1$.

(2) Because $H_\varnothing$ is horospherical, its normalizer $N_G(H_\varnothing)$ in $G$ is
a parabolic subgroup of $G$, and we have a locally trivial fibration  
$G/H_\varnothing \to G/N_G(H_\varnothing)$ for the Zariski topology, 
with fiber isomorphic to $N_G(H_\varnothing)/H_\varnothing$. 
The algebraic torus $N_G(H_\varnothing)/H_\varnothing$  
has dimension the rank $r$ of $G/H$.
Therefore,
\begin{align} \label{eq:poincare}
\tP_{G/H_\varnothing}(t) = (t-1)^{r} \tP_{G/N_G(H_\varnothing)}(t).
\end{align}
Note that $\tP_{G/N_G(H_\varnothing)}(0)=1$. 
On the other hand, by Theorem \ref{Th:Brion-Peyre}, 
$\tP_{G/H}(t)$ is equivalent to 
$ (-1)^{r_G-r_H} t^{u_{G}-u_{H}}$ when $t$ goes to $0$ 
since $Q_{G/H}(0)=1$ for connected~$H$. 
So, according to \eqref{eq:poincare} and the assertion (1), we obtain that
$$(-1)^{r} =(-1)^{r_G-r_H}  t^{u_{G}-u_{H}}  t^{- \deg R_\varnothing}$$
since $R_\varnothing(0)=1$, which forces 
$\deg R_\varnothing = u_G -u_H.$ 
To conclude, notice that $u_G=u_\varnothing$ because $H_\varnothing$ is horospherical.

(3) Assume that $H_I$ is connected. By Theorem \ref{Th:Brion-Peyre},  
$\tP_{G/H_v}(t) $ 
is equivalent to  
$ (-1)^{r_G-r_{I}}t^{u_{G}-u_{I}}$ 
  when $t$ goes to $0$. 
So, from the assertion (1) we deduce that 
$$(-1)^{r_G-r_{I}}  t^{u_{G}-u_{I}}  = (-1)^{r_G-r_H} t^{- \deg R_I}  t^{u_{G}-u_{H}}$$
since $R_I(0)=1$, which forces 
$\deg R_I = u_I -u_H.$ 
\end{proof}

\begin{ex} \label{exa:P}
Keep the notations of the main illustrating example, \S\ref{sub:mot}. 
Given $I \subset  S$, let $\Phi_I$ be the root system generated by $I$ 
and write $\Phi_I^+$ for the set of its positive roots with respect to $I$. 
Note that  $\dim G/P_I = \dim P_I^{u} = | \Phi_S^+ \setminus \Phi_I^+ |$, 
where $P_I^{u}$ is the unipotent radical of $P_I$. 
Simple computations show that 
$$R_I (t^{-1}) = \tP_{G/P_I}(t)  t^{- | \Phi_S^+ \setminus \Phi_I^+ |} 
= 
\prod\limits_{\alpha \in \Phi_S^+ \setminus \Phi_I^+}
\left(\frac{t^{{\rm ht}(\alpha) +1} -1}{t^{{\rm ht}(\alpha)}-1} \right) t^{- |  \Phi_S^+ \setminus \Phi_I^+|},$$
where ${\rm ht}(\alpha)$ denotes the height of $\alpha$ for $\alpha \in \Phi^+_S$. 
In particular, we obtain that the right hand side is 
a polynomial in $t^{-1}$. 
\end{ex}

\begin{rem} 
When $R_I$ verifies the statement of Conjecture \ref{conj:P}, 
it would be interesting to find an algorithm for computing the polynomial $R_I$ 
using combinatorial properties of the spherical roots in $I \subset  \Sigma$,  
as in Example~\ref{exa:P}. 
\end{rem}

\begin{ex} \label{exa4:P}
Assume that $G=SL_2\times SL_2\times SL_2$ and 
that $H$ is $SL_2$,  diagonally embedded in $G$. 
The spherical homogeneous space $G/H$ 
admits three spherical roots $s_1,s_2,s_3$
and so its valuation cone has $2^3=8$ faces. By symmetry,
$H_{\{s_1\}} \cong  H_{\{s_2\}} \cong  H_{\{s_2\}}$ and 
$ H_{\{s_1,s_2\}} \cong
H_{\{s_2,s_3\}} \cong  H_{\{s_1,s_3\}}.$
Applying Theorem~\ref{Th:arc}, we can compute the satellites $H_{\{s_1\}}$, 
$H_{\{s_1,s_2\}} $ and $H_{\varnothing}$ and we find that 
\begin{align*}
H_{\{s_1\}} &= 
\left\{ \left( \begin{pmatrix} a_1 & 0 \\
c_1 & a_1^{-1} \end{pmatrix},\begin{pmatrix} a_1 & b_2 \\
0 & a_1^{-1} \end{pmatrix},\begin{pmatrix} a_1 & 0 \\
0 & a_1^{-1} \end{pmatrix}\right) \; ; \; a_1 \in \C^*, c_1,b_2 \in \C \right\} ,&\\
H_{\{s_1,s_2\}} &=
\left\{ \left( \begin{pmatrix} a_1 & 0 \\[0.1em]
c_1 & a_1^{-1} \end{pmatrix},\begin{pmatrix} a_1 & b_2 \\
0 & a_1^{-1} \end{pmatrix},\begin{pmatrix} a_1 & b_2 \\
0 & a_1^{-1} \end{pmatrix}\right)\; ; \; a_1 \in \C^*, c_1,b_2 \in \C \right\} ,&\\[0.1em]
H_{\varnothing} &= 
\left\{ \left(\begin{pmatrix} a_1 & 0 \\
c_1 & a_1 \end{pmatrix},\begin{pmatrix} a_1 & b_2 \\
0 & a_1 \end{pmatrix},\begin{pmatrix} a_1 & b_3 \\
0 & a_1 \end{pmatrix}\right)  \; ; \; c_1,b_2,b_3 \in \C \text{ and }a_1^2=1\right\}.&
\end{align*}
We observe that $G/H_\varnothing$ is a locally trivial fibration 
for the Zariski topology over $(\P^1)^3$ 
with fiber $(\C^*)^3$. So we get 
$$\tP_{G/H_\varnothing}(t) = (t-1)^3(t+1)^3.$$
On the other hand, $H_{\{s_1\}}$ 
and $H_{\{s_1,s_2\}}$ are connected. 
By Theorem \ref{Th:Brion-Peyre}, we get: 
\begin{align*} 
\tP_{G/H_{\{s_1\}}}(t)=
\tP_{G/H_{\{s_1,s_2\}}}(t) = t(t-1)^2 (t+1)^3. 
\end{align*}
In conclusion, 
$R_{\varnothing} (t^{-1})= 1 -t^{-2}$ and 
$R_{\{s_1\}}(t)=R_{\{s_1,s_2\}}(t^{-1}) =  1+t^{-1}$. 
\end{ex}

\begin{ex} \label{exa2:P}
Let $G=SL_{n}$ and let $H$ be a maximal standard Levi factor of semisimple type
$SL_{n-1} \times \C^*$.  The homogeneous space $G/H$ is spherical and admits
only one elementary embedding, up to isomorphism, $X:=\P(\C^n) \times
\P((\C^n)^*) \cong\P^{n-1} \times \P^{n-1} $.
So $G/H$ has a unique satellite $H_\varnothing$ which is horospherical.
The unique closed $G$-orbit of $X$ is
$$X'= \{ ( [x_1 :\cdots : x_n ] ,[y_1 :\cdots : y_n ] )\in\P^{n-1} \times \P^{n-1} \; |\;
\sum_{i=1}^n x_i y_i=0\}.$$
The variety $X'$ is a locally trivial fibration for the Zariski topology 
over $\P^{n-1}$ with
a fiber isomorphic to $\P^{n-2}$. 
In addition, denoting by $H'$ the stabilizer of a point in $X'$,
we know that
$G/H_{\varnothing}$ is a locally trivial fibration for the Zariski topology  
over $G/H' \cong X'$
with fiber $\C^*$.
Hence
$$\tP_{G/H_{\varnothing}}(t)  =(1+t+\cdots +t^{n-1}) (1+t+\cdots+t^{n-2})(t-1).$$
Observing that $\tP_{G/H}= \tP_{X}(t) - \tP_{X'}(t)$, we get 
$$\tP_{G/H}(t) =  (1+t+\cdots+t^{n-1}) t^{n-1}.$$
In conclusion, 
$R_{\varnothing}(t^{-1}) 
= 1 -t ^{-(n-1)}.$
\end{ex}

\begin{thm} \label{Th:P-one} 
Assume that the spherical homogeneous space $G/H$ is of rank one. 
Then $R_\varnothing$ is a polynomial with 
integer coefficients. 
\end{thm}

\begin{proof}
Spherical homogeneous spaces $G/H$ of rank one were classified by Akhiezer
\cite{Ak83} and Brion \cite{Br89}. Such a homogeneous space $G/H$ is either horospherical
or has a wonderful compactification. 
Theorem \ref{Th:P-one} is obvious if $G/H$ is horospherical. 
So we are interested only in those homogeneous spaces $G/H$ of rank 
one that admit  
a wonderful compactification. These spaces are listed in~\cite[Table 30.1]{T}. 
In this case the spherical subgroup $H$ has a unique satellite subgroup  
different from $H$:  this is the horospherical subgroup $H_\varnothing$.
For each homogeneous space $G/H$ from the list in \cite[Table 30.1]{T},
we compute the Poincar\'{e} polynomials
$\tP_{G/H}(t)$ and $\tP_{G/H_\varnothing}(t)$,  
and then the ratio 
$$R_\varnothing(t^{-1}) =  \dfrac{\tP_{G/{H_\varnothing}}(t)}{\tP_{G/H}(t)}.$$
The obtained results are described in Table \ref{tab:P}.  
Our calculations show that the ratio $R_\varnothing(t^{-1})$ is a polynomial in
$t^{-1}$ containing only two terms, with integer coefficients.

Let us explain our computations. 
Let $G/H \hookrightarrow X$ be a wonderful embedding of $G/H$ 
with closed $G$-orbit $X'$. We have: 
\begin{align} \label{eq:P-one}
\tP_X(t) = \tP_{G/H}(t) + \tP_{X'}(t)  
= \tP_{G/H}(t) + \dfrac{\tP_{G/H_\varnothing}(t)}{(t-1)}.
\end{align}
Note that $X'$ is a projective homogeneous space $G/P$,  
where $P$ is certain parabolic subgroup of $G$ such that 
$\dim G/P= \dim G/H -1$. 

In some cases, like in Example \ref{exa2:P}, 
from the knowledge of $X$ and $X'$, we compute 
$\tP_X(t)$, $\tP_{X'}(t)$ and so $\tP_{G/H}(t)$ and 
$\tP_{G/H_\varnothing}(t)$ by (\ref{eq:P-one}). 
We can proceed in this way for cases 1, 3  (which corresponds to 
Example \ref{exa2:P}), 7a, 7b, 10 of Table \ref{tab:P}.

If $H$ is connected, it is sometimes easier to compute directly $\tP_{G/H}(t)$ 
using Theorem \ref{Th:Brion-Peyre} and \cite[Theorem 1(c)]{BPe}, instead of 
computing $\tP_{X'}(t)$. 
Then we get $\tP_{G/H_\varnothing}(t)$ from $\tP_X(t)$ 
and $\tP_{G/H}(t)$ by (\ref{eq:P-one}). 
We can argue in this way for cases 2, 5, 6, 7a, 7b, 9, 10, 11, 13.  

Still if $H$ is connected, it is alternatively possible to 
compute $\tP_{G/H_\varnothing}(t)$ even without the knowledge of $X$. 
Indeed, it is sometimes possible to deduce the conjugacy class of 
the parabolic stabilizer $P$ of a point in $X'$ just by dimension 
reasons. Then we get $\tP_{G/H_\varnothing}(t)$ 
since 
$$\tP_{G/H_\varnothing}(t) = \tP_{G/P}(t) (t-1).$$

Consider for example the case 12 of Table \ref{tab:P} where 
$G= {\bf F}_4$ and $H={\bf B}_4$. 
Then $\dim G/H = 52 - 36 = 16$. So $\dim G/P = 15$ and 
$\dim P = 37$. Hence, for dimension reasons, 
$P$ is conjugate to a parabolic subgroup 
whose semisimple Levi part is either of type either ${\bf B}_3$,  
or of type ${\bf C}_3$. 
In both cases, we get (assuming that $P$ contains the standard Borel 
subgroup $B$) that  
$$\tP_{G/P}(t)= \frac{\tP_{G/B}(t)}{\tP_{P/B}(t)}= \frac{(t^2-1)(t^{6}-1)(t^{8}-1)(t^{12}-1)}
{(t-1)(t^{2}-1)(t^{4}-1)(t^{6}-1)} $$
since ${\bf B}_3$ and ${\bf C}_3$ share the same exponents 
$1,3,5$. 
Hence 
$$\tP_{G/H_\varnothing}(t) =\tP_{G/P}(t) (t-1) = (t^4+1)(t^{12}-1).$$
On the other hand, since $H$ is connected we have by \cite{BPe}: 
$$\tP_{G/H}=\frac{(t^2-1)(t^6-1)(t^8-1)(t^{12}-1) t^{24}}{ 
(t^2-1)(t^4-1)(t^6-1)(t^8-1) t^{16}} = \frac{(t^{12}-1) t^8}{t^4-1}.$$
In conclusion, 
$$R_\varnothing(t^{-1})= \frac{t^8-1}{t^8} =1 - t^{-8}.$$
We can proceed similarly for cases 4, 
8a, 8b, 14, 15. 
\end{proof}

\begin{center}
{\tiny
\begin{table}
\begin{tabular}{l|lllll}
 \hline & & & && \\[-0.125em] 
 n$^\circ$ & $G$ & $H$  & $\tP_{G/H}(t)$ &  $\tP_{G/H_\varnothing}(t)$ &
 $R_{\varnothing}(t^{-1})$  \\ 
\hline
 & & & && \\
1 & $SL_2\times SL_2$ & $SL_2 $ & $t (t^{2}-1)$ & $ (t-1)(1+t)^{2}$ & $1+t^{-1}$ \\
 & & & && \\
2 & $PSL_2\times PSL_2$ & $PSL_2 $ & $t (t^{2}-1)$ &$ (t-1)(1+t)^2$ & $1+t^{-1}$  \\
 & & & && \\
3 & $SL_{n}$ & $S(L_1\times L_{n-1}) $ & $ \frac{t^{n-1} (t^{n}-1)}{t-1}$ &
$\frac{(t^{n-1}-1)(t^{n}-1)}{t-1}$ & $1 -t ^{-(n-1)}$\\
& & & && \\
4 & $PSL_2$ & $PO_2$ & $t^{2}$ & $t^{2}-1$ & $1-t^{-2}$ \\ 
 & & & && \\
5 & $Sp_{2n}$ & $Sp_2 \times Sp_{2n-2}$ &
$\frac{t^{2n-2}(t^{2n}-1)}{t^2-1} $ &
$\frac{(t^{2n-2}-1)(t^{2n} -1)}{t^2-1}$ & $1-t^{-(2n-2)}$ \\ 
 & & & && \\
6 & $Sp_{2n}$ & $B(Sp_2) \times Sp_{2n-2}$ & $\frac{t^{2n-1}(t^{2n}-1)}{t-1}$ &
$\frac{(t^{2n-1}-1)(t^{2n}-1)}{t-1}$ & $1-t^{-(2n-1)}$ \\
 & & & && \\
7a & $SO_{2n+1}$ & $SO_{2n}$ & $t^{n}(t^{n}+1)$ & $t^{2n}-1$& $1-t^{-n}$  \\ 
 & & & && \\
7b & $SO_{2n}$ & $SO_{2n-1}$ & $t^{n-1}(t^{n}-1) $ & $(t^{n-1}+1)(t^{n}-1)$ & $1+t^{-(n-1)}$ \\
 & & & && \\
8a & $SO_{2n+1}$ & $S(O_1\times O_{2n})$ & $t^{2n}$  & $t^{2n}-1$ & $1-t^{-2n}$ \\
 & & & && \\
 8b & $SO_{2n}$ & $S(O_1\times O_{2n-1})$ & $t^{n-1}(t^{n}-1)$  &
 $(t^{n-1}+1)(t^{n}-1)$ & $1+ t^{-(n-1)}$ \\
 & & & && \\
9 & $SO_{2n}$ & $GL_{n}\rightthreetimes \wedge^{2} \C^n$ &$t \prod_{i=1}^n(t^{i}+1)$
&$(t-1)\prod_{i=1}^n(t^{i}+1)$ & $1-t^{-1}$ \\
 & & & && \\
 10 & $Spin_7$ & ${\bf G}_2$ & $t^{3}(t^{4}-1)$ & $(t^3+1)(t^4-1)$ & $1+t^{-3}$ \\
  & & & && \\
11 & $SO_7$ & ${\bf G}_2$ & $t^{3}(t^{4}-1)$ &
$(t^3+1)(t^4-1)$ & $1+t^{-3}$  \\ 
& & & && \\
12 & ${\bf F}_4$ & ${\bf B}_4$  & $\frac{t^{8}(t^{12}-1)}{t^{4}-1}$
& $(t^4+1)(t^{12}-1)$ & $1 - t^{-8}$ \\ 
 & & & && \\
13 & ${\bf G}_2$ & $SL_3$ & $t^3(t^3+1)$ &
$(t^{3}+1)(t^3-1)$ & $1-t^{-3}$ \\
& & & && \\
14 & ${\bf G}_2$ & $N(SL(3))$ & $t^{6}$ & $t^{6}-1$ & $1-t^{-6}$ \\ 
& & & && \\
15 & ${\bf G}_2$ & $GL_2\rightthreetimes (\C\oplus \C^2)
\otimes \wedge^{2} \C^2$ & $ \frac{t^2(t^6-1)}{t-1}$ &
$\frac{(t^2-1)(t^6-1)}{t-1}$ & $1-t^{-2}$ \\[0.5em]
\hline
\end{tabular}
\bigskip
\caption{Data for homogeneous spherical spaces of rank one} \label{tab:P}
\end{table}}
\end{center}


\vspace{-1cm}
\begin{thebibliography}{............}
\bibitem{Ak83} D. N. Akhiezer,
{\em Equivariant completion of homogeneous algebraic varieties by
homogeneous divisors},
Ann. Global Anal. Geom. {\bf 1} (1983), n$^\circ$1, 49-78.

\bibitem{AM} V. Alexeev and M. Brion,
{\em Toric degenerations of spherical varieties},
Selecta Math. {\bf 10} (2004), n$^\circ$4, 453-478.

\bibitem{Avdeev2013} R. Avdeev, 
{\em Normalizers of solvable spherical subgroups}, 
Mathematical Note {\bf 94} (2013), n$^\circ$1-2, 20-31.

\bibitem{Avdeev2015} R. Avdeev, 
{\em Strongly solvable spherical subgroups and their combinatorial invariants}, 
Selecta Math. {\bf 21} (2015), n$^\circ$3, 931-993. 

\bibitem{BatyrevDais} V. Batyrev and D. I. Dais,  
{\em Strong McKay correspondence, string-theoretic Hodge
numbers and mirror symmetry}, 
Topology {\bf 35} (1996), 901-929.

\bibitem{BM} V. Batyrev and A. Moreau,
{\em The arc space of horospherical varieties and motivic
integration}, Compos. Math.  {\bf 149} (2013), n$^\circ$8,
1327-1352.

\bibitem{BM2} V. Batyrev and A. Moreau,
{\em Stringy $E$-functions for spherical varieties},
in preparation.

\bibitem{BifetConciniProcesi} E. Bifet, C. De Concini and C. Procesi, 
{\em Cohomology of regular embeddings}, 
Adv. Math. {\bf 82} (1990), 1-34. 

\bibitem{Bravi-Pezzini2014} P. Bravi and G. Pezzini, 
{\em Wonderful subgroups of reductive groups and spherical systems}, 
J. Algebra {\bf 409} (2014), 101-147.

\bibitem{Bravi-Pezzini2015} P. Bravi and G. Pezzini, 
{\em The spherical systems of the wonderful reductive subgroups},  
J. Lie Theory {\bf 25} (2015), n$^\circ$1, 105-123.

\bibitem{Bravi-Pezzini2016} P. Bravi and G. Pezzini, 
{\em Primitive wonderful varieties}, 
Math. Z. {\bf 282} (2016), n$^\circ$3-4, 1067-1096.

\bibitem{Br89} M. Brion,
{\em On spherical homogeneous spaces of rank one},
Group actions and invariant theory, CMS Conf. Proc., vol. 10,
31-41, 1989.

\bibitem{B90} M. Brion,
{\em Vers une g\'en\'eralisation des espaces sym\'etriques},
J. Algebra {\bf 134} (1990), n$^\circ$1, 115-143.

\bibitem{Br97} M. Brion,
{\em Vari\'{e}t\'{e}s sph\'{e}riques},
Notes de la session de la S.~M.~F. ^^ ^^ Op\'{e}rations hamiltoniennes et
op\'{e}rations de groupes alg\'{e}briques" (Grenoble, 1997).

\bibitem{BLV} M. Brion, D. Luna and T. Vust,
{\em Espaces homog\`{e}nes sph\'{e}riques},
Invent. Math. {\bf 84} (1986), n$^\circ$3, 617-632.

\bibitem{BP} M. Brion and F. Pauer,
{\em Valuations des espaces homog\`{e}nes sph\'{e}riques},
Comment. Math. Helv. {\bf 62} (1987), n$^\circ$2, 265-285.

\bibitem{BPe} M. Brion and E. Peyre,
{\em The virtual Poincar\'{e} polynomials of homogeneous spaces},
Compos. Math. {\bf 134} (2002), 319-335.

\bibitem{Concini-Procesi83} 
C. De Concini, C. Procesi, 
{\em Complete symmetric varieties}, 
Invariant theory (Montecatini, 1982), 1-44, Lecture Notes in Math. 
{\bf 996}, Springer, Berlin, 1983.

\bibitem{Cupit-Foutou} 
S.~Cupit-Foutou, 
{\em Wonderful varieties:~a geometrical realization}, 
arXiv:0907.2852. 

\bibitem{DanilovKhovanskii}
V.I. Danilov and A.A. Khovanskii, 
{\em Newton polyhedra and an algorithm for computing
Hodge-Deligne numbers}, 
Math. USSR Izv. {\bf 29} (1987), 279-298.

\bibitem{DL97} A. Dimca and G. Lehrer,
{\em Purity and equivariant weight polynomials}, 161-181,
in: Algebraic groups and Lie groups,
Austral. Math. Soc. Lect. Ser., {\bf 9}, Cambridge Univ. Press, Cambridge, 1997.

\bibitem{GH} G. Gagliardi and J. Hofscheier,
{\em Homogeneous spherical data of orbits in spherical embeddings},
Transform. Groups {\bf 20} (2015), n$^\circ$1, 83-98.

\bibitem{GN} D. Gaitsgory and D. Nadler,
{\em Spherical varieties and Langlands duality},
Mosc. Math. J. {\bf 10} (2010), n$^\circ$1, 65-137, 271.

\bibitem{Kav} K. Kaveh,
{\em SAGBI bases and degeneration of spherical varieties to toric varieties},
Michigan Math. J. {\bf 53} (2005), n$^\circ$1, 109-121.

\bibitem{K} F. Knop,
{\em The Luna-Vust theory of spherical embeddings},
in: Proceedings of the Hyderabad Conference on Algebraic Groups
(Hyderabad, 1989) (Madras), Manoj Prakashan (1991), 225-249.

\bibitem{K96} F. Knop,
{\em Automorphisms, root systems, and compactifications of homogeneous varieties},
J. Amer. Math. Soc. {\bf 9} (1996), n$^\circ$1, 153-174.

\bibitem{Losev2009} I. Losev,
{\em Uniqueness property for spherical homogeneous spaces},
Duke Math. J. {\bf 147} (2009), n$^\circ$2, 315-343.


\bibitem{Luna96} D. Luna,
{\em Toute vari\'{e}t\'{e} magnifique est sph\'{e}rique}, 
Transform. Groups {\bf 1} (1996), n$^\circ$3, 249-258. 

\bibitem{Luna97} D. Luna,
{\em Grosses cellules pour les vari\'{e}t\'{e}s sph\'{e}riques},
267-280, Austral. Math. Soc. Lect. Ser., {\bf 9}, Cambridge Univ. Press,
Cambridge, 1997.

\bibitem{Luna01} D. Luna,
{\em Vari\'{e}t\'{e}s sph\'{e}riques de type $A$}, 
Publ. Math. Inst. Hautes \'{E}tudes Sci. No. {\bf 94} (2001), 161-226.

\bibitem{LV} D. Luna and T. Vust,
{\em Plongements d'espaces homog{\`e}nes}, Comment. Math. Helv.,
{\bf 58} (1983), 186-245.

\bibitem{Pa} D.I. Panyushev,
{\em On the conormal bundle of a G-stable subvariety}, 
Manuscripta Math. {\bf 99} (1999), no. 2, 185-202.


\bibitem{Popov87} V.L. Popov, 
{\em Contractions of actions of reductive algebraic groups}, Mat. Sb. (N.S.) 
{\bf 130}(172) (1986), n$^\circ$3, 310-334, {\bf 431}; Math. USSR-Sb. 
{\bf 58} (1987), 
n$^\circ$2, 311-335. 

\bibitem{Vust74} T. Vust, 
{\em Op\'{e}ration de groupes r\'{e}ductifs dans un type de c\^{o}nes presque homog\`{e}nes},
Bull. Soc. Math. France {\bf 102} (1974), 317-333.

\bibitem{Vust90} T. Vust, 
{\em Plongements d'espaces sym\'{e}triques alg\'{e}briques: une classification},  
Ann. Scuola Norm. Sup. Pisa Cl. Sci. (4), {\bf 17} (1990), n$^\circ$2, 165-195. 

\bibitem{T} D. Timashev,
{\em Homogeneous spaces and equivariant embeddings},
Invariant Theory and Algebraic Transformation Groups, {\bf 8}. Springer,
Heidelberg, 2011.
\end{thebibliography}
\end{document}